\documentclass[oneside,english]{amsart}
\usepackage[T1]{fontenc}
\usepackage[latin9]{inputenc}
\usepackage{geometry}
\usepackage{amstext}
\usepackage{amsthm}
\usepackage{amssymb}

\makeatletter
\numberwithin{equation}{section}
\numberwithin{figure}{section}
\theoremstyle{plain}
\newtheorem{thm}{\protect\theoremname}
  \theoremstyle{plain}
  \newtheorem{prop}[thm]{\protect\propositionname}
  \theoremstyle{plain}
  \newtheorem{lem}[thm]{\protect\lemmaname}
  \theoremstyle{plain}
  \newtheorem{cor}[thm]{\protect\corollaryname}
  \theoremstyle{remark}
  \newtheorem{rem}[thm]{\protect\remarkname}

\usepackage{xypic}
\usepackage[pdf]{pstricks}
\usepackage{pst-node,pstricks-add}

\makeatother

\usepackage{babel}
  \providecommand{\corollaryname}{Corollary}
  \providecommand{\lemmaname}{Lemma}
  \providecommand{\propositionname}{Proposition}
  \providecommand{\remarkname}{Remark}
\providecommand{\theoremname}{Theorem}

\begin{document}

\title{Affine lines in the complement of a smooth plane conic }

\author{Julie Decaup }

\address{Institut de Mathématiques de Toulouse, Université Paul Sabatier,
118 route de Narbonne, 31062 Toulouse Cedex 9, France }

\email{jdecaup@math.univ-toulouse.fr }

\author{Adrien Dubouloz}

\address{IMB UMR5584, CNRS, Univ. Bourgogne Franche-Comté, F-21000 Dijon,
France.}

\email{adrien.dubouloz@u-bourgogne.fr}

\thanks{This research was partialy funded by ANR Grant \textquotedbl{}BirPol\textquotedbl{}
ANR-11-JS01-004-01. }

\begin{abstract}
We classify closed curves isomorphic to the affine line in the complement
of a smooth rational projective plane conic $Q$. Over a field of
characteristic zero, we show that up to the action of the subgroup
of the Cremona group of the plane consisting of birational endomorphisms
restricting to biregular automorphisms outside $Q$, there are exactly
two such lines: the restriction of a smooth conic osculating $Q$
at a rational point and the restriction of the tangent line to $Q$
at a rational point. In contrast, we give examples illustrating the
fact that over fields of positive characteristic, there exist exotic
closed embeddings of the affine line in the complement of $Q$. We
also determine an explicit set of birational endomorphisms of the
plane whose restrictions generates the automorphism group of the complement
of $Q$ over a field of arbitrary characteristic. 
\end{abstract}

\maketitle

\section*{Introduction}

A famous theorem of Abhyankar and Moh \cite{AM75} asserts that over
a field $k$ of characteristic zero, all closed embeddings of the
affine line $\mathbb{A}_{k}^{1}$ into the affine plane $\mathbb{A}_{k}^{2}$
are equivalent under the action of the group $\mathrm{Aut}_{k}(\mathbb{A}_{k}^{2})$
of algebraic $k$-automorphisms of $\mathbb{A}_{k}^{2}$: for any
two such closed embeddings with images $A$ and $A'$, there exists
$\Psi\in\mathrm{Aut}_{k}(\mathbb{A}_{k}^{2})$ such that $A'=\Psi(A)$.
In this article, we consider the classification of equivalence classes
of closed embeddings of the affine line into another smooth affine
surface very similar to the affine plane: the complement of a smooth
$k$-rational conic $Q$ in the projective plane $\mathbb{P}_{k}^{2}$. 

In the complex case, such a smooth affine surface $S=\mathbb{P}_{\mathbb{C}}^{2}\setminus Q$
has divisor class group $\mathrm{Cl}(S)=\mathbb{Z}_{2}$, integral
homology groups $H_{0}(S;\mathbb{Z})=\mathbb{Z}$, $H_{1}(S;\mathbb{Z})=\mathbb{Z}_{2}$
and $H_{i}(S;\mathbb{Z})=0$ for every $i\geq2$, and its logarithmic
Kodaira dimension $\overline{\kappa}(S)=\overline{\kappa}(\mathbb{P}^{2},K_{\mathbb{P}^{2}}+Q)$
(see \cite{Ii77}) is equal to $-\infty$. It is thus very close to
the affine plane from both algebraic and topological points of view.
It also contains many closed curves isomorphic to the affine line
$\mathbb{A}^{1}$. For instance, for every point $p\in Q$, $Q$ and
twice its tangent line $T_{p}Q$ at $Q$ generate a pencil $\mathcal{P}_{p}\subset\left|\mathcal{O}_{\mathbb{P}^{2}}(2)\right|$
whose members, except for the one $2T_{p}Q$, are smooth conics intersecting
$Q$ with multiplicity $4$ at $p$. The intersections with $S$ of
all members of $\mathcal{P}_{p}$ except $Q$ are thus isomorphic
to $\mathbb{A}^{1}$. The subgroup $\mathrm{Aut}(\mathbb{P}^{2},Q)$
of $\mathrm{Aut}(\mathbb{P}^{2})$ consisting of automorphism preserving
$Q$ acts transitively on $Q$, and for a given point $p_{0}\in Q$,
the action on $\mathcal{P}_{p_{0}}\setminus Q$ of the subgroup $\mathrm{Aut}(\mathbb{P}^{2},Q,p_{0})$
of $\mathrm{Aut}(\mathbb{P}^{2},Q)$ consisting of automorphisms fixing
$p_{0}$ has exactly two orbits: a fixed point $T_{p_{0}}Q$ and its
complement $\mathcal{P}_{p_{0}}\setminus(Q\cup T_{p_{0}}Q)$. Viewing
$\mathrm{Aut}(\mathbb{P}^{2},Q)$ as a subgroup of $\mathrm{Aut}(S)$
via the natural restriction homomorphism, it follows in particular
that $\mathrm{Aut}(S)$ acts on the set of so-defined affine lines
in $S$ with at most two orbits: the one of $T_{p_{0}}Q\cap S$ and
the one of $Q_{1}\cap S$ for a fixed member $Q_{1}$ of $\mathcal{P}_{p_{0}}\setminus(Q\cup T_{p_{0}}Q)$.
But since $\mathrm{Cl}(S\setminus(S\cap T_{p_{0}}Q_{1}))$ is trivial
while $\mathrm{Cl}(S\setminus(S\cap Q_{1}))\simeq\mathbb{Z}_{2}$,
it follows that $T_{p_{0}}Q\cap S$ and $Q_{1}\cap S$ cannot belong
to a same orbit of the action of $\mathrm{Aut}(S)$ on the set of
closed curves in $S$ isomorphic to $\mathbb{A}^{1}$. So in contrast
with the case of $\mathbb{A}_{\mathbb{C}}^{2}$ , the best we can
hope for is that the action of $\mathrm{Aut}(S)$ on the set of such
closed curves has precisely two orbits. Our main result just below
implies that this is exactly the case: 
\begin{thm}
\label{thm:MainThm} Let $k$ be a field of characteristic zero, let
$Q\subset\mathbb{P}_{k}^{2}$ be a smooth conic and let $S=\mathbb{P}_{k}^{2}\setminus Q$.
Suppose that $A\subset S$ is a closed curve isomorphic to $\mathbb{A}_{k}^{1}$
and let $\overline{A}\subset\mathbb{P}_{k}^{2}$ be its closure. Then
$Q$ is $k$-rational and for every given $k$-rational point $p_{0}\in Q$
and every smooth $k$-rational member $Q_{1}\neq Q$ of the pencil
$\mathcal{P}_{p_{0}}$ generated by $Q$ and $2T_{p_{0}}Q$, there
exists a birational map $\Phi:\mathbb{P}_{k}^{2}\dashrightarrow\mathbb{P}_{k}^{2}$
defined over $k$, restricting to an automorphism of $S$, such that
\[
\Phi_{*}(\overline{A})=\begin{cases}
Q_{1} & \textrm{if }\deg\overline{A}\textrm{ is even}\\
T_{p_{0}}Q & \textrm{if }\deg\overline{A}\textrm{ is odd}.
\end{cases}
\]
In particular, there are precisely two classes of closed curves isomorphic
to $\mathbb{A}_{k}^{1}$ in $S$ up to the action of $\mathrm{Aut}_{k}(S)$. 
\end{thm}
Note that the dichotomy depending on the degree of $\overline{A}$
follows from the fact that the divisor classes of $Q$ and $\overline{A}$
either generate $\mathrm{Cl}(\mathbb{P}_{k}^{2})$ when $\deg\overline{A}$
is odd, or generate a proper subgroup of index $2$ when $\deg\overline{A}$
is even, so that $\mathrm{Cl}(S\setminus A)=\left\{ 0\right\} $ or
$\mathbb{Z}_{2}$ according to $\deg\overline{A}$ is odd or even.
\\

Recall that by a theorem of Jung and van der Kulk \cite{Ju42,vdK53},
the automorphism group of the affine plane $\mathbb{A}_{k}^{2}=\mathrm{Spec}(k[x,y])$
over an arbitrary field $k$ is the free product of the group of affine
automorphisms and of the group of automorphisms of the form $(x,y)\mapsto(ax+b,cy+s(x))$,
where $a,c\in k^{*}$, $b\in k$ and $s\in k[t]$, amalgamated over
their intersection. Viewing $\mathbb{A}_{k}^{2}$ as the complement
of the line at infinity $L=\left\{ z=0\right\} $ in $\mathbb{P}_{k}^{2}=\mathrm{Proj}(k[x,y,z])$,
these two subgroups coincide respectively with the restriction to
$\mathbb{A}_{k}^{2}$ of the group $\mathrm{Aut}(\mathbb{P}_{k}^{2},L)$
and with the group $\mathrm{Aut}(\mathbb{A}_{k}^{2},\mathrm{pr}_{1})$
of automorphisms preserving globally the $\mathbb{A}^{1}$-fibration
$\mathrm{pr}_{1}:\mathbb{A}_{k}^{2}\rightarrow\mathbb{A}_{k}^{1}$
induced by the restriction of the pencil of lines through the point
$[0:1:0]$. Our second result consists of an analogous presentation
of the automorphism group of the complement of a smooth $k$-rational
conic in $\mathbb{P}_{k}^{2}$, providing in particular a complete
description of the birational maps $\Phi:\mathbb{P}_{k}^{2}\dashrightarrow\mathbb{P}_{k}^{2}$
which can occur in Theorem \ref{thm:MainThm}. 
\begin{thm}
\label{thm:MainTh2} Let $k$ be a field of arbitrary characteristic
and let $S$ be the complement of a smooth conic $Q\subset\mathbb{P}_{k}^{2}$
with a $k$-rational point $p$. Let $\mathrm{Aut}(S,\rho_{p})$ denote
the subgroup of $\mathrm{Aut}_{k}(S)$ consisting of automorphisms
which preserve globally the $\mathbb{A}^{1}$-fibration $\rho_{p}:S\rightarrow\mathbb{A}_{k}^{1}$
induced by restriction of the pencil $\mathcal{P}_{p}\subset\left|\mathcal{O}_{\mathbb{P}_{k}^{2}}(2)\right|$
generated by $Q$ and twice its tangent line $T_{p}Q$ at $p$. Then
$\mathrm{Aut}_{k}(S)$ is the free product of $\mathrm{Aut}(\mathbb{P}_{k}^{2},Q)|_{S}$
and $\mathrm{Aut}(S,\rho_{p})$ amalgamated along their intersection.
\end{thm}
The scheme of the article is the following: in the first section,
we review standard material on projective completions of smooth quasi-projective
surfaces and certain rational fibrations on them. Section two is devoted
to the proof of Theorem \ref{thm:MainThm}, which proceeds through
the analysis of the structure of the total transform of the divisor
$Q\cup\overline{A}$ in a minimal log-resolution of the pair $(\mathbb{P}_{k}^{2},Q\cup\overline{A})$.
Our argument, inspired by a recent alternative proof of the Abhyankar-Moh
and Lin-Zaidenberg theorems due to Palka \cite{Pal15}, uses techniques
and classification results from the theory of $\mathbb{Q}$-\emph{acyclic
complex surfaces}, that is, normal complex surfaces with trivial reduced
rational homology groups. A standard reference for most of these results
is \cite{MiyBook}, to which we refer the reader for a more complete
picture of the theory of non complete algebraic surfaces. Theorem
\ref{thm:MainTh2} is proved in the third section, in which we give
in addition an explicit set of generators of $\mathrm{Aut}_{k}(S)$
for a suitably chosen model of $Q$ up to projective equivalence.
We also derive from this description examples illustrating that similarly
to the situation for the affine plane, Theorem \ref{thm:MainThm}
does not hold over fields of positive characteristic.\\

\subsection*{Acknowledgments }

Some of the questions addressed in this article emerged during the
workshop ``Birational geometry of surfaces'', held at the Department
of Mathematics of the University of Roma Tor Vergata in January 2016.
The authors would like to thank the organizers of the workshop for
the motivated but relaxed atmosphere of this workshop, as well as
the other members of the ``Afternoon Cremona Club'', Ciro Cilliberto,
Alberto Calabri and Anne Lonjou, for stimulating discussions. 

\section{Preliminaries and notations}

In what follows, the term $k$-variety refers to a geometrically integral
scheme $X$ of finite type over a base field $k$ of arbitrary characteristic.
A morphism of $k$-varieties is a morphism of $k$-schemes. A surface
$V$ is a $k$-variety of dimension $2$, and by a curve on a surface,
we mean a geometrically reduced closed sub-scheme $C\subset V$ of
pure codimension $1$ defined over $k$.

\subsection{SNC divisors and smooth completions}

\indent\newline\indent (i) An \emph{SNC divisor} on a smooth projective
surface $X$ is a curve $B$ on $X$ with smooth irreducible components
and ordinary double points only as singularities. Equivalently, for
every closed point $p\in B$, the local equations of the irreducible
components of $B$ passing through $p$ form a part of regular sequence
in the maximal ideal $\mathfrak{m}_{S,p}$ of the local ring $\mathcal{O}_{S,p}$
of $S$ at $p$. 

An SNC divisor $B$ on $X$ is said to be \emph{SNC-minimal} if there
does not exist any strictly birational projective morphism $\tau:X\rightarrow X'$
onto a smooth projective surface $X'$ with exceptional locus contained
in $B$ such that $\tau_{*}(B)$ is SNC. If $k$ is algebraically
closed, then this property is equivalent to the fact that any $(-1)$-curve
$E$ contained in $B$ is \emph{branching}, i.e. meets at least three
other irreducible components of $B$. 

(ii) A \emph{smooth} \emph{completion} of a smooth quasi-projective
surface $V$ is a pair $(X,B)$ consisting of a smooth projective
surface $X$ and an SNC divisor $B\subset V$ such that $X\setminus B\simeq V$. 

\subsection{\label{subsec:Trees-Chains}Rational trees and rational chains}

\indent\newline\indent (i) A \emph{geometrically rational tree} $B$
on a smooth projective surface $X$ is an SNC divisor whose irreducible
components are geometrically rational curves and such that the dual
graph of the base extension $B_{\overline{k}}$ of $B$ to an algebraic
closure $\overline{k}$ of $k$ is a tree. A \emph{geometrically rational
chain} $B$ is a geometrically rational tree such that the dual graph
of $B_{\overline{k}}$ is a chain. A \emph{rational tree} (resp. \emph{rational
chain}) is a geometrically rational tree (resp. geometrically rational
chain) whose irreducible components are all $k$-rational.

The irreducible components $B_{0},\ldots,B_{r}$ of a rational chain
$B$ can be ordered in such a way that $B_{i}\cdot B_{j}=1$ if $|i-j|=1$
and $0$ otherwise. A rational chain $B$ with such an ordering on
the set of its irreducible components is said to be \emph{oriented}.
The components $B_{0}$ and $B_{r}$ are called respectively the left
and right boundaries of $B$, and we say by extension that an irreducible
component $B_{i}$ of $B$ is on the left of another one $B_{j}$
when $i<j$. The sequence of self-intersections $[B_{0}^{2},\ldots,B_{r}^{2}]$
is called the \emph{type} of the oriented rational chain $B$. An
\emph{oriented} \emph{subchain} of an oriented rational chain $B$
is a rational chain $Z$ whose support is contain in that of $B$.
We say that an oriented rational chain $B$ is composed of subchains
$D_{1},\ldots,D_{s}$, and we write $B=D_{1}\vartriangleright\cdots\vartriangleright D_{s}$,
if the $D_{i}$ are oriented subchains of $B$ whose union is $B$
and the irreducible components of $D_{i}$ precede those of $D_{j}$
for $i<j$. 

(ii) An oriented rational chain $F\vartriangleright C\vartriangleright E$
where $F$ and $C$ are irreducible and $E$ is an oriented subchain,
possibly empty, is said to be $m$-\emph{standard, $m\in\mathbb{Z}$,}
if it is of type $[0,-m]$ or $[0,-m,-a_{1},\ldots,-a_{r}]$ where
$a_{i}\geq2$ for every $i=1,\ldots,r$. It is an elementary exercise
(see e.g. \cite{Gi-Da1}) to check that every chain $B$ with non
negative definite intersection matrix can be transformed by a sequence
of blow-ups and blow-downs whose centers are contained in the successive
total transforms of $B$ either into a $0$-curve, or into a chain
of type $[0,0,0]$, or into an $m$-standard chain for every $m\in\mathbb{Z}$. 

(iii) In particular, every affine surface $S$ non isomorphic to $\mathbb{A}_{k}^{1}\times(\mathbb{A}_{k}^{1}\setminus\{0\})$
admitting a smooth completion $(X_{0},B_{0})$ for which $B_{0}$
is a rational chain, admits a smooth completion $(X,B)$ for which
$B=F\vartriangleright C\vartriangleright E$ is $m$-standard chain
(see e.g. \cite[Lemma 2.7]{Du05}). Furthermore, it follows from a
result of Danilov and Gizatullin \cite[Corollary 2]{Gi-Da1} that
the type of the subchain $E$ is an invariant of $S$, in the sense
that if $(X',B')$ is another smooth completion of $S$ by an $m'$-standard
chain $B'=F'\vartriangleright C'\vartriangleright E'$ then the type
of $E'$ is either equal to that of $E$ or to that of $E$ equipped
with the reversed orientation. For instance, if $S=\mathbb{P}_{k}^{2}\setminus Q$
is the complement of a smooth $k$-rational conic $Q$ in $\mathbb{P}^{2}$,
then for every smooth completion $(X,B)$ of $S$ by a $m$-standard
chain $B=F\vartriangleright C\vartriangleright E$, the subchain $E$
has type $[-2,-2,-2]$. 

\subsection{\label{subsec:Fibrations}Recollection on $\mathbb{P}^{1}$, $\mathbb{A}^{1}$
and $\mathbb{A}_{*}^{1}$ -fibrations }

We review some basic properties of $\mathbb{P}^{1}$-fibrations on
smooth projective surfaces and their restrictions to certain of their
open subsets, see e.g \cite[Chapter 3]{MiyBook} for more details. 

(i) By a $\mathbb{P}^{1}$-fibration on a smooth projective surface
$X$, we mean a surjective morphism $\overline{\rho}:X\rightarrow\overline{Z}$
onto a smooth projective curve $\overline{Z}$ whose generic fiber
is isomorphic to the projective line over the function field of $\overline{Z}$.
It is well known that every $\mathbb{P}^{1}$-fibration $\overline{\rho}:X\rightarrow\overline{Z}$
is obtained from a Zariski locally trivial $\mathbb{P}^{1}$-bundle
over $\overline{Z}$ by a finite sequence of blow-ups of points. In
particular, every such $\mathbb{P}^{1}$-fibration has a section,
and its singular fibers are supported by geometrically rational trees
on $X$. If $X$ is $k$-rational, then it follows from the Riemann-Roch
Theorem, that for every smooth $k$-rational curve $F$ with self-intersection
$0$, the complete linear system $|F|$ defines a $\mathbb{P}^{1}$-fibration
$\overline{\rho}_{|F|}:X\rightarrow\mathbb{P}_{k}^{1}$ having $F$
as a smooth fiber. 

(ii) An $\mathbb{A}^{1}$-fibration on a smooth quasi-projective surface
$V$ is a surjective morphism $\rho:V\rightarrow Z$ onto a smooth
curve $Z$ whose generic fiber is isomorphic to the affine line over
the function field of $Z$. Every $\mathbb{A}^{1}$-fibration is the
restriction of a $\mathbb{P}^{1}$-fibration $\overline{\rho}:X\rightarrow\overline{Z}$
over the smooth projective model $\overline{Z}$ of $Z$ on a smooth
completion $(X,B)$ of $V$. Furthermore, one can always find such
a smooth completion for which $B$ has the form $B=\bigcup_{z\in\overline{Z}\setminus Z}F_{z}\cup C\cup\bigcup_{z\in Z}H_{z}$
where, $F_{z}=\overline{\rho}^{-1}(z)\simeq\mathbb{P}_{\kappa(z)}^{1}$
for every $z\in\overline{Z}\setminus Z$, $C$ is a section of $\overline{\rho}$,
and where for every $z\in Z$, $H_{z}$ is an SNC-minimal geometrically
rational subtree of $\overline{\rho}^{-1}(z)$, possibly empty, the
support of the fiber $\overline{\rho}^{-1}(z)$ being equal to the
union of $H_{z}$ and of the closure in $X$ of the support of $\rho^{-1}(z)$. 

If in addition $V$ is affine, then every nonempty $H_{z}$ contains
a $\kappa(z)$-rational irreducible component intersecting $C$, the
closure in $X$ of every irreducible component of $\rho^{-1}(z)$
is isomorphic to the projective line over a finite extension $\kappa'$
of $\kappa(z)$, and it intersects $H_{z}$ transversally in a unique
$\kappa'$-rational point. A scheme theoretic closed fiber $\rho^{-1}(z)$
of $\rho:V\rightarrow Z$ which is not isomorphic to $\mathbb{A}_{\kappa(z)}^{1}$
is said to be \emph{degenerate}. 

(iii) An $\mathbb{A}_{*}^{1}$-fibration on smooth quasi-projective
surface $V$ is a surjective morphism $\xi:V\rightarrow Z$ onto a
smooth affine curve $Z$ whose geometric generic fiber is isomorphic
to the punctured affine line $\mathbb{A}_{*}^{1}=\mathbb{A}^{1}\setminus\{0\}$
over an algebraic closure of the function field of $Z$. We say that
$\xi$ is twisted if the generic fiber of $\xi$ is a nontrivial form
of $\mathbb{A}_{*}^{1}$ over the function field of $Z$, and untwisted
otherwise.

\section{Affine lines in $\mathbb{P}^{2}\setminus Q$ }

This section is devoted to the proof of Theorem \ref{thm:MainThm}.
A field $k$ of characteristic zero being fixed throughout, we let
$S$ be the complement of a smooth conic $Q$ in $\mathbb{P}_{k}^{2}$,
we let $A\subset S$ be a closed curve isomorphic to $\mathbb{A}_{k}^{1}$
and we let $\overline{A}$ be its closure in $\mathbb{P}^{2}$. Let
$\nu:C\rightarrow\overline{A}$ be the normalization of $\overline{A}$.
Since $A\simeq\mathbb{A}_{k}^{1}$, $C$ is isomorphic to $\mathbb{P}_{k}^{1}$
and $C\setminus\nu^{-1}(A)$ consists of a unique point. So $\overline{A}\setminus A$
consists of a unique point $p=Q\cap\overline{A}$, which is thus necessarily
$k$-rational, at which $\overline{A}$ has a unique local analytic
branch. In particular, $Q$ is $k$-rational. Since $\mathrm{Aut}(\mathbb{P}_{k}^{2})$
acts transitively on the set of pairs $(Q,p)$ consisting of a smooth
$k$-rational conic and a $k$-rational point on it and since the
stabilizer $\mathrm{Aut}(\mathbb{P}_{k}^{2},Q,p)\subset\mathrm{Aut}(\mathbb{P}_{k}^{2})$
of a given pair $(Q,p)$ acts transitively on $Q$, we are reduced
to establish the following:
\begin{prop}
\label{prop:Intermediate-Formulation} There exists a smooth $k$-rational
conic $Q'\subset\mathbb{P}_{k}^{2}$ and a birational map $\Psi:\mathbb{P}_{k}^{2}\dashrightarrow\mathbb{P}_{k}^{2}$
restricting to an isomorphism $\psi:S=\mathbb{P}_{k}^{2}\setminus Q\stackrel{\sim}{\rightarrow}\mathbb{P}_{k}^{2}\setminus Q'$
mapping $\overline{A}$ either to a smooth $k$-rational conic intersecting
$Q'$ in a unique $k$-rational point $p'$ or to the tangent line
$T_{p'}Q'$ of $Q'$ at a $k$-rational point $p'$. 
\end{prop}
The proof of this proposition if given in $\S$ \ref{subsec:LogRes}
and $\S$ \ref{subsec:ProofThm} below.

\subsection{\label{subsec:LogRes} Log-resolution setup and preliminary observations }

Let $\sigma:(X',D')\rightarrow(\mathbb{P}_{k}^{2},Q\cup\overline{A})$
be the minimal log-resolution of the pair $(\mathbb{P}_{k}^{2},Q\cup\overline{A})$.
Recall that by definition, $X'$ is smooth, $\sigma$ is a projective
birational morphism restricting to an isomorphism over $\mathbb{P}_{k}^{2}\setminus Q\cup\overline{A}$,
and minimal for the property that $D'=\sigma^{-1}(Q\cup\overline{A})_{\mathrm{red}}$
is an SNC divisor. Note that if $p=Q\cap\overline{A}$ is a singular
point of $\overline{A}$ then $\sigma$ is in particular a log-resolution
of the singularity of $\overline{A}$. Since $\overline{A}\cdot Q\geq2$
and $p$ is $k$-rational, $\sigma$ consists of the blow-up of $p$
followed by a sequence of blow-ups of $k$-rational points supported
on the successive total transforms of $Q\cup\overline{A}$. Since
$\overline{A}$ has a unique analytic branch at $p$, $D'$ is a rational
tree of the form $\overline{A}\cup D{}_{1}'\cup E'\cup D_{2}'$, where
$D_{1}'\cup E'\cup D_{2}'=\sigma^{-1}(Q)_{\mathrm{red}}$ consists
of a rational tree $D_{1}'$ containing the proper transform of $Q$,
a nonempty SNC-minimal rational chain $D_{2}'$ with negative definite
 intersection matrix and a $(-1)$-curve $E'$ such that $D_{1}'\cap E'$,
$D_{2}'\cap E'$ and $\overline{A}'\cap E'$ all consist of a unique
point.

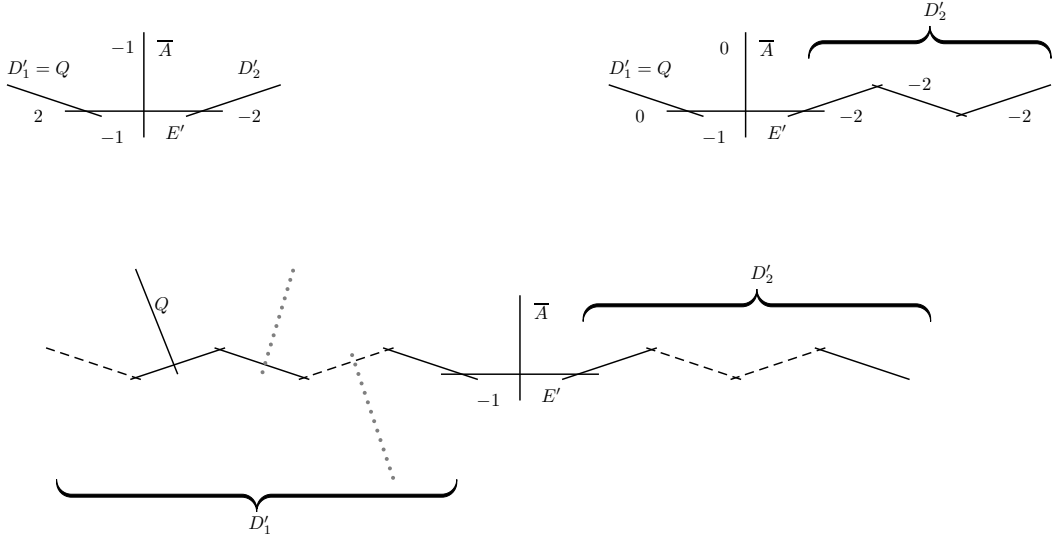
\begin{figure}[!htb]
\psset{linewidth=0.8pt}
\begin{pspicture}(-2,2)(10,-5.5)

\rput(-1,0){
\psscalebox{0.7 0.7}{

\psline(0,-0.5)(0,1.5)
\psline(-1.5,0)(1.5,0)
\psline(0.8,-0.1)(2.6,0.5)
\psline(-0.8,-0.1)(-2.6,0.5)
\rput(0.6,-0.4){$E'$}
\rput(-0.6,-0.5){$-1$}
\rput(0.4,1.2){$\overline{A}$}
\rput(-0.4,1.2){$-1$}
\rput(2,0.8){$D_2'$}
\rput(2,-0.1){$-2$}
\rput(-2,0.8){$D_1'=Q$}
\rput(-2,-0.1){$2$}
}
}
\rput(7,0){
\psscalebox{0.7 0.7}{
\psline(-1.5,0)(1.5,0)
\psline(0,-0.5)(0,1.5)
\psline(-0.8,-0.1)(-2.6,0.5)
\psline(0.8,-0.1)(2.6,0.5)
\psline(2.4,0.5)(4.2,-0.1)
\psline(4,-0.1)(5.8,0.5)
\rput(0.6,-0.4){$E'$}
\rput(-0.6,-0.5){$-1$}
\rput(0.4,1.2){$\overline{A}$}
\rput(-0.4,1.2){$0$}
\rput(2,-0.1){$-2$}
\rput(3.3,0.5){$-2$}
\rput(5.2,-0.1){$-2$}
\psbrace[linewidth=0.01,rot=-90, nodesepB=-3pt, nodesepA=-3.5pt](5.8,1)(1.2,1){$D_2'$}
\rput(-2,0.8){$D_1'=Q$}
\rput(-2,-0.1){$0$}
}
}
\rput(4,-3.5){
\psscalebox{0.7 0.7}{
\psline(-1.5,0)(1.5,0)
\psline(0,-0.5)(0,1.5)
\psline(0.8,-0.1)(2.6,0.5)
\psline[linestyle=dashed](2.4,0.5)(4.2,-0.1)
\psline[linestyle=dashed](4,-0.1)(5.8,0.5)
\psline(5.6,0.5)(7.4,-0.1)
\psline(-0.8,-0.1)(-2.6,0.5)
\psline[linestyle=dashed](-2.4,0.5)(-4.2,-0.1)
\psline(-4.0,-0.1)(-5.8,0.5)
\psline(-5.6,0.5)(-7.4,-0.1)
\psline[linestyle=dashed](-7.2,-0.1)(-9.0,0.5)
\psline[linestyle=dotted, linewidth=2.2pt, linecolor=gray](-4.9,0)(-4.3,2) 
\psline[linestyle=dotted, linewidth=2.2pt, linecolor=gray](-3.2,0.4)(-2.4,-2) 
\psline(-6.5,0)(-7.3,2)
\psbrace[linewidth=0.01,rot=-90, nodesepB=-3pt, nodesepA=-3.5pt](7.8,1)(1.2,1){$D_2'$}
\psbrace[linewidth=0.01,rot=90, nodesepB=12pt, nodesepA=-4.5pt](-8.8,-2)(-1.2,-2){$D_1'$}
\rput(-6.8,1.3){$Q$}
\rput(0.6,-0.4){$E'$}
\rput(-0.6,-0.5){$-1$}
\rput(0.4,1.2){$\overline{A}$}
}
}
\end{pspicture}
\caption{Structure of the divisor $D'$ in the case where $\overline{A}$ is a line, a rational conic, and a general curve respectively. The gray dotted lines represent rational subtrees of $D_1'$.}
\label{fig:reso1}
\end{figure}

The proper transform of $Q$ is the unique possible non-branching
$(-1)$-curve in $D_{1}'$ and we let $\sigma':(X',D')\rightarrow(X,D)$
be the map consisting of the contraction all successive non-branching
$(-1)$-curves in $D_{1}'$ . The image of $D'$ by $\sigma'$ is
again a rational tree $D=\tilde{A}\cup D_{1}\cup E\cup D_{2}$ where
$\tilde{A}=\sigma_{*}(\overline{A})\simeq\mathbb{P}_{k}^{1}$, $D_{1}=\sigma'_{*}(D_{1})$
is an SNC-minimal rational tree, $D_{2}=\sigma_{*}(D_{2}')$ is a
rational chain isomorphic to $D_{2}$, and $E=\sigma'_{*}(E')$. By
construction $S=\mathbb{P}_{k}^{2}\setminus Q$ is isomorphic to $X\setminus(D_{1}\cup E\cup D_{2})$. 

We now establish two crucial auxiliary results which will serve for
the analysis of the case where the self-intersection of the proper
transform $\tilde{A}$ of $\overline{A}$ in $X$ is negative. 
\begin{lem}
\label{lem:Negative} If $\tilde{A}^{2}<0$ then the following hold: 

a) The rational tree $D_{1}$ is not empty, and its intersection matrix
is not negative definite. 

b) Every closed irreducible curve $C$ in $X$ distinct from $\tilde{A}$
or an irreducible component of $D_{2}$ intersects $D_{1}$.
\end{lem}
\begin{proof}
These properties are invariant under extension and restriction of
the base field $k$. So by first replacing $k$ by a subfield $k_{0}\subset k$
of finite transcendence degree over $\mathbb{Q}$ over which the projective
surface $X$, the divisor $D=\tilde{A}\cup D_{1}\cup E\cup D_{2}$
and the curve $C$ are defined and then taking base extension by an
embedding $k_{0}\hookrightarrow\mathbb{C}$, we may assume from the
beginning that $k=\mathbb{C}$. Then since $S\simeq X\setminus(D_{1}\cup E\cup D_{2})\simeq\mathbb{P}^{2}\setminus Q$
is $\mathbb{Q}$-acyclic, the classes $E$ and of the irreducible
components of $D_{1}$ and $D_{2}$ form a basis of $\mathrm{Cl}(X)\otimes_{\mathbb{Z}}\mathbb{Q}$
\cite[Lemma 4.2.1]{MiyBook}. In $\mathrm{Cl}(X)\otimes_{\mathbb{Z}}\mathbb{Q}$,
we may thus write $\tilde{A}\equiv(\tilde{A}^{2})E+R$, where $R$
is the class of a $\mathbb{Q}$-divisor supported on $D_{1}\cup D_{2}$
and since $\tilde{A}^{2}\neq0$, it follows that the classes of $\tilde{A}$
and of the irreducible components of $D_{1}$ and $D_{2}$ also form
a basis of $\mathrm{Cl}(X)\otimes_{\mathbb{Z}}\mathbb{Q}$. Let $\tau:X\rightarrow\tilde{X}$
be the birational morphism onto a normal surface with at most cyclic
quotient singularities obtained by contracting $\tilde{A}$ and the
negative definite rational chain $D_{2}$. The irreducible components
of $\tau(D_{1})$ then form a basis of $\mathrm{Cl}(\tilde{X})\otimes_{\mathbb{Z}}\mathbb{Q}$,
and since $\tau(D_{1})$ is connected, it follows that $\tilde{X}\setminus\tau(D_{1})$
is a normal $\mathbb{Q}$-acyclic surface, hence an affine surface
by virtue \cite{Fu82}. So by \cite{Go69}, $\tau(D_{1})$ is the
support of an ample divisor on $\tilde{X}$, in particular, $\tau(D_{1})$
is not empty, its intersection matrix is not negative definite and
it intersects every proper curve in $\tilde{X}$. Because $D_{1}$
is disjoint from $D_{2}$ and $\tilde{A}$, $\tau$ restricts to an
isomorphism in an open neighborhood of $D_{1}$, and so the assertion
follows. 
\end{proof}
\begin{lem}
\label{lem:Asquare}If $E^{2}=-1$ and $\tilde{A}^{2}<0$ then $\tilde{A}^{2}=-1$. 
\end{lem}
\begin{proof}
Similarly as in the proof of the previous lemma, the assertion is
invariant under restriction to a subfield $k_{0}\subset k$ of finite
transcendence degree over $\mathbb{Q}$ over which $X$ and $D$ are
defined and extension to $\mathbb{C}$ via the choice of an embedding
$k_{0}\hookrightarrow\mathbb{C}$. So we may again assume that $k=\mathbb{C}$.
The following argument is inspired from \cite{Pal15}. Suppose for
contradiction $\tilde{A}^{2}<-1$. Then $\Delta=D_{1}\cup\tilde{A}\cup D_{2}$
is an SNC-minimal divisor on $X$ with three connected components,
whose complement is a smooth quasi-projective surface $V$ such that
$V\setminus(E\cap V)\simeq S\setminus A$. The theory of minimal models
of log-surfaces (see \cite[Chapter 3]{MiyBook} for a detailed account)
asserts the existence of a sequence of projective birational morphisms
\begin{equation}
f=f_{n}\circ\cdots\circ f_{1}:(X,\Delta)=(X_{0},\Delta_{0})\stackrel{f_{1}}{\longrightarrow}(X_{1},\Delta_{1})\stackrel{f_{2}}{\longrightarrow}\cdots\stackrel{f_{n}}{\longrightarrow}(X_{n},\Delta_{n})\label{eq:Minimalization}
\end{equation}
with the following properties: 

a) For every $i=1,\ldots,n$, $X_{i}$ is a smooth projective surface
and $\Delta_{i}$ is an SNC-minimal reduced divisor

b) For every $i=1,\ldots,n$, $f_{i}:(X_{i-1},\Delta_{i-1})\rightarrow(X_{i},\Delta_{i})$
is the contraction of a $(-1)$-curve $\ell_{i}\not\subset\Delta_{i-1}$
intersecting $\Delta_{i-1}$ transversally in at most two smooth points
and each connected component of $\Delta_{i-1}$ at most once and such
that $\overline{\kappa}(X_{i},\Delta_{i})=\overline{\kappa}(X_{i},\Delta_{i}\cup\ell_{i})$,
followed by the SNC-minimalization $\Delta_{i}$ of the push-forward
$(f_{i})_{*}\Delta_{i-1}$ of $\Delta_{i}$. Furthermore, if $\ell_{i}$
intersects precisely two connected components of $\Delta_{i-1}$ then
one of these components is rational chain with negative definite intersection
matrix. 

c) The pair $(X_{n},\Delta_{n})$ is \emph{almost-minimal}, meaning
that every log Minimal Model Program ran from $(X_{n},\Delta_{n})$
terminates with a log terminal projective surface $(X_{n+1},\Delta_{n+1})$
and the induced birational morphism $f_{n+1}:(X_{n},\Delta_{n})\rightarrow(X_{n+1},\Delta_{n+1}=(f_{n+1})_{*}\Delta_{n})$
contracts only irreducible components of $\Delta_{n}$ onto quotient
singularities of $X_{n+1}$. Equivalently, $(X_{n+1},\Delta_{n+1})$
is a log-terminal pair which does not contain any proper curve $\ell$
such that $\ell^{2}<0$ and $\ell\cdot(K_{X_{n+1}}+\Delta_{n+1})<0$
and $f_{n+1}:(X_{n},\Delta_{n})\rightarrow(X_{n+1},\Delta_{n+1})$
is its minimal log-resolution. 

Since $\Delta_{0}=D_{1}\cup\tilde{A}\cup D_{2}$ has three connected
components, $\Delta_{n}$ has at most three connected components,
and since by Lemma \ref{lem:Negative} a) the intersection matrix
of $D_{1}$ is not negative definite, the connected component of $\Delta_{n}$
containing the image of $D_{1}$ is not contracted to a point by $f_{n+1}$.
So the exceptional locus $\mathrm{Exc}(f_{n+1})$ consists of at most
two connected components of $\Delta_{n}$, and since $\Delta_{n}$
is SNC-minimal, $f_{n+1}(\mathrm{Exc}(f_{n+1}))$ consists of singular
points of $X_{n+1}$. In particular, the local fundamental group $G_{p}$
at every point $p\in f_{n+1}(\mathrm{Exc}(f_{n+1}))$ has order at
least $2$. An elementary calculation shows that the topological Euler
characteristic of the surface $X_{i-1}\setminus\Delta_{i-1}$ increases
at a step if and only if the curve $\ell_{i}$ contracted by $f_{i}$
intersects two connected components of $\Delta_{i-1}$ and the union
of $\ell_{i}$ with these components is contracted by $f_{i}$ to
a smooth point of $X_{i}$. If such a curve existed, then by Lemma
\ref{lem:Negative} b), one of these connected components would necessarily
be the one containing the image of $D_{1}$, which would imply in
turn that the intersection matrix of $D_{1}$ is definite negative,
a contradiction to Lemma \ref{lem:Negative} a). So for every $i=1,\ldots,n$,
\[
\chi(X_{i}\setminus\Delta_{i})\leq\chi(V)=\chi(S\setminus A)+\chi(E\cap V)=-1.
\]

Now suppose that $\overline{X}\setminus\Delta$ has non-negative logarithmic
Kodaira dimension $\overline{\kappa}(X,\Delta)\geq0$. Then $\overline{\kappa}(X_{n},\Delta_{n})\geq0$
and since $(X_{n},\Delta_{n})$ is almost-minimal, it follows from
the logarithmic Bogomolov-Miyaoka-Yau inequality \cite{La03} that
\[
0\leq\chi(X_{n}\setminus\Delta_{n})+\sum_{p\in f_{n+1}(\mathrm{Exc}(f_{n+1}))}\frac{1}{|G_{p_{i}}|}\leq\chi(X_{n}\setminus\Delta_{n})+\frac{1}{2}\pi_{0}(\mathrm{Exc}(f_{n+1})).
\]
The only possibility is thus that $\chi(X_{n}\setminus\Delta_{n})=\chi(V)=-1$
and that $f_{n+1}(\mathrm{Exc}(f_{n+1}))$ consists of two points,
with local fundamental group $\mathbb{Z}_{2}$. The corresponding
connected components of $\Delta_{n}$ are thus simply $(-2)$-curves,
and we deduce in turn from Lemma \ref{lem:Negative} a) that $D_{2}$
and $\tilde{A}$ are $(-2)$-curves themselves. Since $E^{2}=-1$
by hypothesis, the complete linear system $|D_{2}+2E+\tilde{A}|$
on $X$ defines a $\mathbb{P}^{1}$-fibration $\overline{\xi}:X\rightarrow\mathbb{P}^{1}$
having the irreducible component $D_{1,E}$ of $D_{1}$ intersecting
$E$ as a $2$-section. The restriction of $\xi$ to $S=X\setminus(D_{1}\cup E\cup D_{2})$
is thus a twisted $\mathbb{A}_{*}^{1}$-fibration $\xi:S\rightarrow Z$
over an open subset $Z$ of $\mathbb{P}^{1}$. Since $S$ is a $\mathbb{Q}$-acyclic,
it follows from \cite[Lemma 4.5.1]{MiyBook} that $Z=\mathbb{A}^{1}$.
The fiber $\overline{\xi}$ over the point $\infty=\mathbb{P}^{1}\setminus Z$
is supported by a connected component $F_{\infty}$ of $D_{1}-D_{1,E}$,
and since $D_{1}$ is a rational tree, $D_{1,E}$ intersects $F_{\infty}$
transversally in a unique point. Since $D_{1}$ is SNC-minimal, we
infer that $F_{\infty}=F_{\infty,1}\vartriangleright L\vartriangleright F_{\infty,2}$
is a chain of type $[-2,-1,-2]$ intersecting $D_{1,E}$ along $L$.
Since $S=\mathbb{P}^{2}\setminus Q$ admits a smooth completion by
a rational curve, it follows from \cite[Theorem 2.16]{Du05} that
by contracting successively $E$, $D_{2}$ and then all successive
non-branching $(-1)$-curves in $D_{1}$, the image of $D$ in the
corresponding smooth projective surface $Y$ is an SNC-minimal chain
$B$ such that $Y\setminus B\simeq S$. Since the irreducible components
$F_{\infty,1}$ and $F_{\infty,2}$ are untouched during this process,
$B$ must be equal to the image of $F_{\infty}$, which is a chain
of type $[-2,a,-2]$ for some $a\geq0$. But one checks that no such
chain can be transformed into one of type $[0,-1,-2,-2,-2]$, a contradiction
to $\S$ \ref{subsec:Trees-Chains} (iii). 

So $\overline{\kappa}(X,\Delta)=-\infty$, and hence $\overline{\kappa}(X_{n},\Delta_{n})=-\infty$.
Since $\chi(X_{n}\setminus\Delta_{n})\leq-1$ and the connected component
of $\Delta_{n}$ containing the image of $D_{1}$ is not contracted
by $f_{n+1}$, it follows from Theorem 3.15.1 and Theorem 5.1.2 in
\cite{MiyBook} that $X_{n}\setminus\Delta_{n}$ $\mathbb{A}^{1}$-ruled,
i.e. contains a Zariski open subset of the form $C\times\mathbb{A}^{1}$
for a certain smooth rational curve $C$. It follows in turn that
$V$ is $\mathbb{A}^{1}$-ruled, and we let $q:V\dashrightarrow\mathbb{P}^{1}$
and $\overline{q}:X\dashrightarrow\mathbb{P}^{1}$ be the rational
maps induced by the projection $\mathrm{pr}_{C}$. By virtue of Lemma
\ref{lem:Negative} b), the closure $F$ in $X$ of a general fiber
of $q$ intersects $D_{1}$. So $q:V\rightarrow Z$ is a well defined
$\mathbb{A}^{1}$-fibration over an open subset $Z$ of $\mathbb{P}^{1}$,
and if $\overline{q}$ is not regular, then its unique proper base
point is supported on $D_{1}$. Furthermore, $D_{2}$ and $\tilde{A}$
are necessarily contained in fibers of $\overline{q}:X\dashrightarrow\mathbb{P}^{1}$.
So $\overline{q}$ is well defined on $S$, restricting to an $\mathbb{A}^{1}$-fibration
$\rho:S\rightarrow\mathbb{A}^{1}$ containing $A$ in one of its fibers.
Let $\delta:Y\rightarrow X$ be the minimal resolution of the indeterminacies
of $\overline{q}$, so that $\overline{\rho}=\overline{q}\circ\sigma:Y\rightarrow\mathbb{P}^{1}$
is an everywhere defined $\mathbb{P}^{1}$-fibration, say with section
$C\subset\delta^{-1}(D_{1})$. Since the closures in $X$ of the general
fibers of $q$ intersect $D_{1}$, the proper transform in $Y$ of
$D_{2}\cup E\cup\tilde{A}$ is contained in $\overline{\rho}^{-1}(\rho(A))$
and since by \cite[Theorem 4.3.1]{MiyBook}, all fibers of $\rho:S\rightarrow\mathbb{A}^{1}$
are irreducible, $\overline{\rho}^{-1}(\rho(A))$ is the union of
the proper transform of $D_{2}\cup E\cup\tilde{A}$, the proper transform
of a subset of irreducible components of $D_{1}$, possibly empty,
and a subset of exceptional divisors of $\sigma$, again possibly
empty. Since $D_{2}$ is a nonempty chain with negative definite intersection
matrix and $\tilde{A}^{2}<-1$, at the step of the contraction of
$\overline{\rho}^{-1}(\rho(A))$ onto a smooth fiber of a $\mathbb{P}^{1}$-fibration,
the image of $E$ would have to become a $(-1)$-curve intersecting
the images of $D_{2}$ and $\tilde{A}$, and either an irreducible
component of the image of $\overline{\rho}^{-1}(\rho(A))$ or the
image of the section $C$, which is impossible. This is absurd, and
so $\tilde{A}^{2}=-1$ necessarily. 
\end{proof}

\subsection{\label{subsec:ProofThm} Proof of Proposition \ref{prop:Intermediate-Formulation}}

\indent\newline\noindent Proposition \ref{prop:Intermediate-Formulation},
and hence Theorem \ref{thm:MainThm}, are now consequences of the
following lemma: 
\begin{lem}
\label{prop:twoCases}With the notation of \S \ref{subsec:LogRes}
above, the following alternative holds: 

1) Either $\tilde{A}^{2}=0$ and then there exists a birational map
$\psi:X\dashrightarrow\mathbb{P}^{2}$ restricting to an isomorphism
between $S=X\setminus D$ and the complement of a smooth conic $Q'$
and mapping $\tilde{A}$ to a smooth conic intersecting $Q'$ with
multiplicity $4$ at a single $k$-rational point $p'$. 

2) Or $\tilde{A}^{2}=-1$ and then there exists a birational map $\psi:X\dashrightarrow\mathbb{P}^{2}$
restricting to an isomorphism between $S=X\setminus D$ and the complement
of a smooth conic $Q'$ and mapping $\tilde{A}$ to the tangent line
to $Q'$ at a $k$-rational point $p'$. 
\end{lem}
\begin{proof}
We consider the two cases $\tilde{A}^{2}\geq0$ and $\tilde{A}^{2}<0$
separately.

Case 1). $\tilde{A}^{2}\geq0$. Suppose first that $a=\tilde{A}^{2}>0$.
Then by performing a minimal sequence of blow-ups of $k$-rational
points supported on the successive proper transforms of $\tilde{A}$,
starting with that of $E\cap\tilde{A}$ and then continuing with those
of the intersection point of the proper transform of $\tilde{A}$
with the previous exceptional divisor produced, we obtain a surface
$\beta:Y\rightarrow X$ in which the proper transform of $\tilde{A}$
has self-intersection $0$, while its reduced total transform is a
rational chain $\tilde{A}\vartriangleright H\vartriangleright D_{3}$
where $H$ is a $(-1)$-curve and $D_{3}$ is a chain of $a-1$ curves
with self-intersection $-2$ connecting $H$ to $E$. The complete
linear system $|\tilde{A}|$ then defines a $\mathbb{P}^{1}$-fibration
$\overline{q}:Y\rightarrow\mathbb{P}_{k}^{1}$ having $\tilde{A}$
as a smooth fiber and $H$ as a section. 

\begin{figure}[!htb]
\psset{linewidth=0.8pt}
\begin{pspicture}(-2,0.6)(10,-5)

\rput(-1,0){
\psscalebox{1 0.8}{
\psline(-1,-4)(-1,0.2)
\rput(-1.4,-1.5){$\tilde{A}$}
\rput(-0.8,-1.55){$0$}
\psline(-1.5,0)(2.5,0)
\rput(0.8,0.3){$H$}
\rput(0.8,-0.3){$-1$}
\psline[linestyle=dashed,linewidth=0.05pt, linecolor=gray ](-0.5,-4)(-0.5,0.2)
\psline[linestyle=dashed,linewidth=0.05pt, linecolor=gray ](-0,-4)(-0,0.2)
\psline[linestyle=dashed,linewidth=0.05pt, linecolor=gray ](0.5,-4)(0.5,0.2)
\psline[linestyle=dashed,linewidth=0.05pt, linecolor=gray ](1,-4)(1,0.2)
\psline[linestyle=dashed,linewidth=0.05pt, linecolor=gray ](1.5,-4)(1.5,0.2)
\pnode(2,0.2){D3a}
\pnode(2.5,-1.5){D3b}
\ncline{D3a}{D3b}\ncput*{{\small $D_3$}}
\psline(2.5,-1.2)(2,-2.2)
\rput(2,-1.4){$E$}
\pnode(2,-2){D2a}
\pnode(2.5,-3.7){D2b}
\ncline{D2a}{D2b}\ncput*{{\small $D_2$}} 
\pnode(2.1,-1.75){D1a}
\pnode(3.8,-1.75){D1b}
\ncline{D1a}{D1b}\ncput*{{\small $D_1$}} 
\pnode(0.8,-4.5){qa}
\pnode(0.8,-5.5){qb}
\ncline{->}{qa}{qb}\naput{{\small $\overline{q}$}}
\psline(-1.5,-6)(2.8,-6)
\rput(3,-6){{\small $\mathbb{P}^1_k$}}
}
}

\rput(7,0){
\psscalebox{1 0.8}{
\psline(-1,-4)(-1,0.2)
\rput(-1.4,-1.5){$\tilde{A}$}
\rput(-0.8,-1.55){$0$}
\psline(-1.5,0)(2.5,0)
\rput(0.8,0.3){$E$}
\psline[linestyle=dashed,linewidth=0.05pt, linecolor=gray ](-0.5,-4)(-0.5,0.2)
\psline[linestyle=dashed,linewidth=0.05pt, linecolor=gray ](-0,-4)(-0,0.2)
\pnode(0.5,-4){D2a}
\pnode(0.5,0.2){D2b}
\ncline{D2a}{D2b}\ncput*{{\small $D_2$}} 
\psline[linestyle=dashed,linewidth=0.05pt, linecolor=gray ](1,-4)(1,0.2)
\psline[linestyle=dashed,linewidth=0.05pt, linecolor=gray ](1.5,-4)(1.5,0.2)
\pnode(2,-4){D1a}
\pnode(2,0.2){D1b}
\ncline{D1a}{D1b}\ncput*{{\small $D_1$}} 

\pnode(0.8,-4.5){rhoa}
\pnode(0.8,-5.5){rhob}
\ncline{->}{rhoa}{rhob}\naput{{\small $\overline{\rho}$}}
\psline(-1.5,-6)(2.8,-6)
\rput(3,-6){{\small $\mathbb{P}^1_k$}}
}
}
\end{pspicture}
\caption{The $\mathbb{P}^1$-fibrations $\overline{q}:Y\rightarrow \mathbb{P}^1_k$ and $\overline{\rho}:X\rightarrow \mathbb{P}^1_k$ respectively.}
\label{fig:PAFib-posAsquare}
\end{figure}
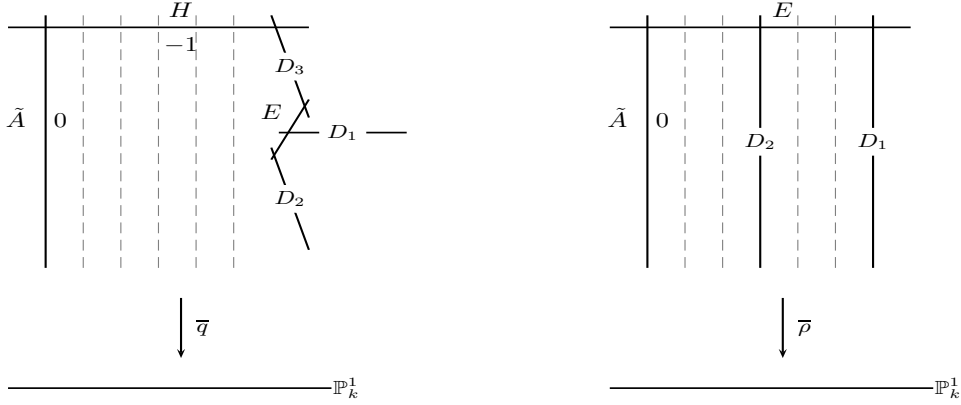The surface $Y$, the rational chain $\tilde{A}\vartriangleright H\vartriangleright D_{3}$
and the $\mathbb{P}^{1}$-fibration are all defined over a subfield
$k_{0}\subset k$ of finite transcendence degree over $\mathbb{Q}$,
and applying \cite[Theorem 4.3.1]{MiyBook} to the complex surface
obtained by base change via an embedding $k_{0}\hookrightarrow\mathbb{C}$,
we conclude that the restriction of $\overline{q}$ to $S\simeq Y\setminus\beta^{-1}(D)$
is an $\mathbb{A}^{1}$-fibration $q:S\rightarrow\mathbb{A}_{k}^{1}$
with a unique degenerate fiber. Since the union of the proper transform
of $D_{1}\cup E\cup D_{2}$ with $D_{3}$ is connected, it would be
fully contained in a unique degenerate fiber of $\overline{q}:Y\rightarrow\mathbb{P}_{k}^{1}$
hence equal to it, and $q$ would be an $\mathbb{A}^{1}$-fibration
without any degenerate fiber, a contradiction. 

So $\tilde{A}^{2}=0$, and the $\mathbb{P}^{1}$-fibration $\overline{\rho}:X\rightarrow\mathbb{P}_{k}^{1}$
defined by the complete linear system $|\tilde{A}|$ restricts to
an $\mathbb{A}^{1}$-fibration $\rho:S\rightarrow\mathbb{A}_{k}^{1}$
having $\tilde{A}\cap S\simeq A$ as a fiber. Since $E$ is a section
of $\overline{\rho}$, for the same reason as before, either $D_{1}$
or $D_{2}$ supports a full fiber $F$ of $\overline{\rho}:X\rightarrow\mathbb{P}_{k}^{1}$,
and since the intersection matrix of $D_{2}$ is negative definite
and $D_{1}$ is SNC minimal, it must be that $F=D_{1}$ is a $(0)$-curve.
So $D=D_{1}\vartriangleright E\vartriangleright D_{2}$ is a $(-E^{2})$-standard
chain, and by \S \ref{subsec:Trees-Chains} (iii), $D_{2}$ is thus
a chain of type $[-2,-2,-2]$. After performing elementary transformations
with center on $E$ if necessary to reach a smooth completion $(X',B')$
of $S$ by a rational chain $F_{\infty}\vartriangleright E\vartriangleright D_{2}$
of type $[0,-1,-2,-2,-2]$, the images of $F_{\infty}$ and $\tilde{A}$
by the contraction $\tau:X'\rightarrow X''$ of $E\cup D_{2}$ are
curves of self-intersection $4$ intersecting each others with multiplicity
$4$ in a single $k$-rational point. Since $X''\setminus\tau(F_{\infty})\simeq S$
and $\mathrm{Cl}(S)\simeq\mathbb{Z}_{2}$, $\mathrm{Cl}(X'')\otimes_{\mathbb{Z}}\mathbb{Q}$
is freely generated by the class of $\tau(F_{\infty})$, and since
$X''$ is a smooth $k$-rational surface, we conclude that $X''\simeq\mathbb{P}_{k}^{2}$
and that $\tau(F_{\infty})$ and $\tau(\tilde{A})$ are smooth $k$-rational
conics.

Case 2). $\tilde{A}^{2}<0$. By Lemma \ref{lem:Negative} a), $D_{1}$
is not empty. We consider two subcases according to the self-intersection
of $E$.

$\quad$ Subcase 1). $E^{2}=-1$. By virtue of Lemma \ref{lem:Asquare},
$\tilde{A}^{2}=-1$. It follows that the complete linear system $|E+\tilde{A}|$
on $X$ defines a $\mathbb{P}^{1}$-fibration $\overline{\xi}:X\rightarrow\mathbb{P}_{k}^{1}$
having the irreducible components $D_{1,E}$ and $D_{2,E}$ of $D_{1}$
and $D_{2}$ intersecting $E$ as disjoint sections. The restriction
of $\overline{\xi}$ to $S=X\setminus(D_{1}\cup E\cup D_{2})$ is
thus an untwisted $\mathbb{A}_{*}^{1}$-fibration $\xi:S\rightarrow Z$
over a smooth curve $Z\subset\mathbb{P}_{k}^{1}$ having $A=\tilde{A}\cap S$
as a degenerated fiber of multiplicity $1$. Let again $k_{0}\subset k$
be a subfield of finite transcendence degree over $\mathbb{Q}$ over
which $X$, $D$ and $\overline{\xi}$ are defined, denote by $X_{\mathbb{C}}$,
$D_{\mathbb{C}}$ and $\overline{\xi}_{\mathbb{C}}$ the corresponding
surface, divisor and morphism obtained by base extension via an embedding
$k_{0}\hookrightarrow\mathbb{C}$, and let $S_{\mathbb{C}}=X_{\mathbb{C}}\setminus(D_{1,\mathbb{C}}\cup E_{\mathbb{C}}\cup D_{2,\mathbb{C}})\simeq\mathbb{P}_{\mathbb{C}}^{2}\setminus Q_{\mathbb{C}}$.
It follows from Lemma 4.5.1 and Theorem 4.6.2 in \cite{MiyBook} applied
to $S_{\mathbb{C}}$ that $Z_{\mathbb{C}}=\mathbb{P}_{\mathbb{C}}^{1}$
and that $\xi_{\mathbb{C}}=\overline{\xi}_{\mathbb{C}}\mid_{S_{\mathbb{C}}}$
has at most a second degenerate fiber, whose support $F$ is isomorphic
to $\mathbb{A}_{*}^{1}$. Since $(D_{2,E})_{\mathbb{C}}^{2}=D_{2,E}^{2}\leq-2$
and $H_{1}(S_{\mathbb{C}};\mathbb{Z})=\mathbb{Z}_{2}$ we deduce from
$\S$ 4.5.2 (5) and Theorem 4.6.1 (2) in \cite{MiyBook} that $\overline{\xi}_{\mathbb{C}}^{-1}(\overline{\xi}_{\mathbb{C}}(E_{\mathbb{C}}\cup\tilde{A}))$
is actually the unique degenerate fiber of $\overline{\xi}_{\mathbb{C}}$.
It follows in turn that $D_{1,\mathbb{C}}=(D_{1,E})_{\mathbb{C}}$
and $D_{2,\mathbb{C}}=(D_{2,E})_{\mathbb{C}}$ and hence that $D_{1}=D_{1,E}$
and $D_{2}=D_{2,E}$. Thus $D$ is a chain $D_{1}\vartriangleright E\vartriangleright D_{2}$
of type $[D_{1}^{2},-1,D_{2}^{2}]$, where $D_{2}^{2}\leq-2$ and
where, by virtue of Lemma \ref{lem:Negative} a), $D_{1}^{2}\geq0$
because the intersection matrix of $D_{1}$ is not negative definite.
Such a chain has a $1$-standard form of type $[0,-1,-2,-2,-2]$ if
and only if $D_{1}^{2}=2$ and $D_{2}^{2}=-2$, and then the images
of $D_{1}$ and $\tilde{A}$ by the contraction $\tau:X\rightarrow\mathbb{P}_{k}^{2}$
of $E$ and $D_{2}$ are respectively a smooth $k$-rational conic
$Q'$ and its tangent line $T_{p'}Q'$ at the $k$-rational point
$p'=\tau(E\cup D_{2})$.

$\quad$ Subcase 2). $E^{2}\geq0$. Since $D_{1}$ is SNC-minimal,
$D$ is SNC minimal, and since the boundary of every SNC-minimal completion
of $S$ is a rational chain by virtue of \cite[Theorem 2.16]{Du05},
$D$ is a rational chain. So $D_{1}$ is an SNC-minimal rational chain
with non negative definite intersection matrix, and hence it contains
an irreducible component with non negative self-intersection. If $D_{1}$
is a $(0)$-curve, then the $\mathbb{P}^{1}$-fibration $\overline{q}:X\rightarrow\mathbb{P}_{k}^{1}$
defined by $|D_{1}|$ has $E$ as a section, hence restricts to an
$\mathbb{A}^{1}$-fibration on $S$ containing $A=\tilde{A}\cap S$
in one of its fibers. Since $D_{2}$ and $\tilde{A}$ both intersect
$E$, they are contained in two different fibers of $\overline{q}$.
But since $\tilde{A}^{2}<0$, $\tilde{A}$ would be properly contained
in a degenerate fiber of $\overline{q}$, which is impossible by virtue
of \S \ref{subsec:Fibrations} (ii). So $D_{1}$ is either reducible
or irreducible with positive self-intersection. By elementary birational
transformations whose centers blown-up and curves contracted are $k$-rational
and supported on\textbf{ $D_{1}$ }and its successive images, we obtain
a smooth completion $W$ of $S$ for which the reduced total transform
$\tilde{D}=W\setminus S$ of $D$ is a rational chain $\tilde{D}_{1}\vartriangleright E\vartriangleright D_{2}$,
where the reduced total transform $\tilde{D}_{1}=F_{\infty}\vartriangleright C\vartriangleright D_{3}$
of $D_{1}$ is a $1$-standard chain. The complete linear system $\left|F_{\infty}\right|$
on $Y$ defines a $\mathbb{P}^{1}$-fibration $\overline{\rho}:W\rightarrow\mathbb{P}_{k}^{1}$
with section $C$, containing $D_{3}\cup E\cup D_{2}\cup\tilde{A}$
in one of its fibers. 

\begin{figure}[!htb]
\psset{linewidth=0.8pt}
\begin{pspicture}(-2,0.6)(10,-5)
\rput(3,0){
\psscalebox{1 0.8}{
\psline(-1,-4)(-1,0.2)
\rput(-1.4,-1.5){$F_{\infty}$}
\rput(-0.8,-1.55){$0$}
\psline(-1.5,0)(2.5,0)
\rput(0.8,0.3){$C$}
\rput(0.8,-0.3){$-1$}
\psline[linestyle=dashed,linewidth=0.05pt, linecolor=gray ](-0.5,-4)(-0.5,0.2)
\psline[linestyle=dashed,linewidth=0.05pt, linecolor=gray ](-0,-4)(-0,0.2)
\psline[linestyle=dashed,linewidth=0.05pt, linecolor=gray ](0.5,-4)(0.5,0.2)
\psline[linestyle=dashed,linewidth=0.05pt, linecolor=gray ](1,-4)(1,0.2)
\psline[linestyle=dashed,linewidth=0.05pt, linecolor=gray ](1.5,-4)(1.5,0.2)
\pnode(2,0.2){D3a}
\pnode(2.5,-1.5){D3b}
\ncline{D3a}{D3b}\ncput*{{\small $D_3$}}
\psline(2.5,-1.2)(2,-2.2)
\rput(2,-1.4){$E$}
\pnode(2,-2){D2a}
\pnode(2.5,-3.7){D2b}
\ncline{D2a}{D2b}\ncput*{{\small $D_2$}} 
\psline(2.1,-1.75)(3.3,-1.75)
\rput(3,-2){$\tilde{A}$}
\pnode(0.8,-4.5){qa}
\pnode(0.8,-5.5){qb}
\ncline{->}{qa}{qb}\naput{{\small $\overline{\rho}$}}
\psline(-1.5,-6)(2.8,-6)
\rput(3,-6){{\small $\mathbb{P}^1_k$}}
}
}
\end{pspicture}
\caption{The $\mathbb{P}^1$-fibration $\overline{\rho}:W\rightarrow \mathbb{P}^1_k$. }
\label{fig:PAFib-posAsquare}
\end{figure}So $E^{2}\leq-1$ and since $E$ intersects $D_{2}$, $\tilde{A}$
and either an irreducible component of $D_{3}$ if $D_{3}$ is not
empty or $C$ otherwise, we have $E^{2}\leq-2$ necessarily. The chain
$D_{3}\vartriangleright E\vartriangleright D_{2}$ is thus SNC-minimal,
hence of type $[-2,-2,-2]$ by \S \ref{subsec:Trees-Chains} (iii).
By applying \cite[Theorem 4.3.1]{MiyBook} to the surface $S_{\mathbb{C}}$
defined in a similar way as in the previous subcase, we deduce that
$A_{\mathbb{C}}$ must be the support of the unique degenerate fiber
of the restriction $\rho_{\mathbb{C}}:S_{\mathbb{C}}\rightarrow\mathbb{A}_{\mathbb{C}}^{1}$
of $\overline{\rho}_{\mathbb{C}}$. So $\overline{\rho}_{\mathbb{C}}^{-1}(\rho_{\mathbb{C}}(A_{\mathbb{C}}))_{\mathrm{red}}=D_{3,\mathbb{C}}\cup E_{\mathbb{C}}\cup D_{2,\mathbb{C}}\cup\tilde{A}_{\mathbb{C}}$
implying that $\tilde{A}_{\mathbb{C}}^{2}=-1$. Thus $\tilde{A}^{2}=-1$
and the images of $F_{\infty}$ and $\tilde{A}$ by the contraction
of $C\cup D_{3}\cup E\cup D_{2}$ to a $k$-rational point $p'$ are
then respectively a smooth $k$-rational conic $Q'$ in $\mathbb{P}_{k}^{2}$
and its tangent line $T_{p'}Q'$ at $p'$. 
\end{proof}
As a consequence of Proposition \ref{prop:twoCases} and of the fact
that $\mathrm{Aut}(\mathbb{P}_{k}^{2})$ acts transitively on the
set of pairs $(Q,p)$ consisting of a smooth $k$-rational conic and
a $k$-rational point on it, we obtain the following:
\begin{cor}
Let $k$ be a field of characteristic $0$ and let $S=\mathbb{P}_{k}^{2}\setminus Q$
be the complement of a smooth $k$-rational conic $Q\subset\mathbb{P}_{k}^{2}$.
Then the following hold:

a) Every closed curves $A\simeq\mathbb{A}_{k}^{1}$ is equal to the
support of a fiber of an $\mathbb{A}^{1}$-fibration $\rho:S\rightarrow\mathbb{A}_{k}^{1}$.
More precisely, there exists a smooth completion $(\mathbb{P}_{k}^{2},Q')$
of $S$ by a smooth $k$-rational conic $Q'$ such that $A$ is the
support of a fiber of the $\mathbb{A}^{1}$-fibration induced by the
restriction of the pencil generated by $Q$ and twice its tangent
line at a $k$-rational point $p$.

b) There exists a unique equivalence class of $\mathbb{A}^{1}$-fibrations
$\rho:S\rightarrow\mathbb{A}_{k}^{1}$ on $S$ up to automorphisms,
in the sense that every two such $\mathbb{A}^{1}$-fibrations $\rho:S\rightarrow\mathbb{A}_{k}^{1}$
and $\rho':S\rightarrow\mathbb{A}_{k}^{1}$ fit into a commutative
diagram \[\xymatrix{ S \ar[r]^{\Psi} \ar[d]_{\rho} & S \ar[d]^-{\rho'} \\ \mathbb{A}^1_k \ar[r]^{\psi} & \mathbb{A}^1_k,} \]
for some automorphisms $\Psi$ and $\psi$ of $S$ and $\mathbb{A}_{k}^{1}$
respectively. 
\end{cor}
\begin{rem}
In the complex case, it was established more generally in \cite[Theorem 2.1]{GM08}
that on a smooth $\mathbb{Q}$-acyclic surface $S$ admitting a smooth
completion $(X,B)$ by a chain of rational curves, every closed curve
isomorphic to $\mathbb{A}_{\mathbb{C}}^{1}$ is the support of a fiber
of an $\mathbb{A}^{1}$-fibration $\rho:S\rightarrow\mathbb{A}_{\mathbb{C}}^{1}$.
But every such $\mathbb{Q}$-acyclic surface  different from $\mathbb{A}_{\mathbb{C}}^{2}$
or $\mathbb{P}_{\mathbb{C}}^{2}\setminus Q$ turns out to have more
than one equivalence class of $\mathbb{A}^{1}$-fibrations up to action
of its automorphism groups. Indeed, by \cite[Theorem 5.6]{Du06},
a $\mathbb{Q}$-acyclic surface as above different from $\mathbb{A}_{\mathbb{C}}^{2}$
is isomorphic to the quotient $S_{m,q}$ of a smooth surface $S_{m}=\left\{ xz=y^{m}-1\right\} \subset\mathbb{A}_{\mathbb{C}}^{3}$,
$m\geq2$, by a free action of the group $\mu_{m}$ of complex $m$-th
roots of unity of the form $(x,y,z)\mapsto(\varepsilon x,\varepsilon^{q}y,\varepsilon^{-1}z)$
where $q\in\left\{ 1,\ldots,m-1\right\} $ and $\gcd(m,q)=1$. Note
that for every $m$ and every such $q$, the involution $(x,y,z)\mapsto(z,y,x)$
of $S_{m}$ descends to isomorphism between $S_{m,q}$ and $S_{m,m-q}$.
The $\mathbb{A}^{1}$-fibration $\mathrm{pr}_{x}:S_{m}\rightarrow\mathbb{A}_{\mathbb{C}}^{1}$
descends to an $\mathbb{A}^{1}$-fibration $\rho_{m,q}:S_{m,q}\rightarrow\mathbb{A}_{\mathbb{C}}^{1}$
having $\rho_{m,q}^{-1}(0)$ as unique degenerate fiber, isomorphic
to $\mathbb{A}_{\mathbb{C}}^{1}$, of multiplicity $m$, whose support
$F_{m,q}$ generates the divisor class group $\mathrm{Cl}(S_{m,q})\simeq\mathbb{Z}_{m}$
of $S_{m,q}$. Furthermore, the $\mu_{m}$-invariant regular $2$-form
$x^{m-q}(\frac{dx\wedge dy}{x})$ on $S_{m}$ descends to a regular
$2$-form vanishing at order $m-q$ along $F_{m,q}$ and nowhere else,
implying that $K_{S_{m,q}}\sim(m-q)F_{m,q}$. It follows that if $m\geq3$,
then the $\mathbb{A}^{1}$-fibrations $\rho_{m,q}$ and $\rho_{m,m-q}$
on $S_{m,q}\simeq S_{m,-q}$ are not equivalent under the action of
$\mathrm{Aut}(S_{m,q})$. Indeed, otherwise there would exist an isomorphism
$\Psi:S_{m,q}\stackrel{\sim}{\rightarrow}S_{m,m-q}$ and an automorphism
$\psi$ of $\mathbb{A}_{\mathbb{C}}^{1}$ fixing the origin such that
$\rho_{m,m-q}\circ\Psi=\psi\circ\rho_{m,q}$, and we would have the
relation $(m-q)F_{m,q}=qF_{m,q}$ in $\mathrm{Cl}(S_{m,q})$, in contradiction
with the fact that $F_{m,q}$ has order $m$ in $\mathrm{Cl}(S_{m,q})$. 
\end{rem}

\section{Automorphisms of $\mathbb{P}^{2}\setminus Q$ and exotic affine lines }

In this section, we fix a base field $k$ of arbitrary characteristic
$p\geq0$. Recall that a smooth $k$-rational conic $Q\subset\mathbb{P}_{k}^{2}$
is projectively equivalent to that $Q_{0}\subset\mathbb{P}_{k}^{2}$
defined by the equation $q_{0}=xz+y^{2}=0$ and that the induced action
on $Q_{0}$ of the stabilizer $\mathrm{Aut}(\mathbb{P}_{k}^{2},Q_{0})$
of $Q_{0}$ in $\mathrm{Aut}(\mathbb{P}_{k}^{2})$ is transitive on
the set of $k$-rational points of $Q_{0}$. We let $S_{0}=\mathbb{P}^{2}\setminus Q_{0}$,
$p_{0}=\left[0:0:1\right]$ and we denote by 
\[
\rho_{0}:S_{0}\rightarrow\mathbb{A}_{k}^{1}=\mathrm{Spec}(k[t]),\quad[x:y:z]\mapsto\frac{x^{2}}{q_{0}}
\]
the $\mathbb{A}^{1}$-fibration induced by the restriction to $S_{0}$
of the rational map $\overline{\rho}_{0}:\mathbb{P}_{k}^{2}\dashrightarrow\mathbb{P}_{k}^{1}$
defined by the pencil $\mathcal{P}_{p_{0}}\subset\left|\mathcal{O}_{\mathbb{P}_{k}^{2}}(2)\right|$
generated by $Q_{0}$ and twice its tangent line $T_{p_{0}}Q_{0}$
at $p_{0}$. We denote by $\mathrm{Aut}(S_{0},\rho_{0})$ the group
of $k$-automorphisms of $S_{0}$ preserving $\rho_{0}$ globally,
that is, automorphisms $\Psi\in\mathrm{Aut}_{k}(S_{0})$ for which
there exists $\psi_{\rho_{0}}\in\mathrm{Aut}(\mathbb{A}_{k}^{1})$
such that $\rho_{0}\circ\Psi=\psi_{\rho_{0}}\circ\rho_{0}$.

\subsection{Automorphisms of $S_{0}$ }

This subsection is devoted to the proof of the following more precise
version of Theorem \ref{thm:MainTh2}. 
\begin{prop}
\label{prop:Automorphisms} With the notation above, the following
hold:

1) The group $\mathrm{Aut}_{k}(S_{0})$ is isomorphic to the free
product of $\mathrm{Aut}(\mathbb{P}_{k}^{2},Q_{0})\mid_{S_{0}}$ and
$\mathrm{Aut}(S_{0},\rho_{0})$ amalgamated along their intersection. 

2) The group $\mathrm{Aut}(\mathbb{P}_{k}^{2},Q_{0})\subset\mathrm{PGL}_{3}(k)$
is isomorphic to $\mathrm{PGL}_{2}(k)$, generated by the following
automorphisms:

$\quad$ a) $[x:y:z]\mapsto[x:y+bx:z-2by-b^{2}x]$, $b\in k$,

$\quad$ b) $[x:y:z]\mapsto[ax:y:a^{-1}z]$, $a\in k^{*}$, 

$\quad$ c) $[x:y:z]\mapsto[z:-y:x]$. 

3) The group $\mathrm{Aut}(S_{0},\rho_{0})$ is generated by the restrictions
to $S_{0}$ of birational endomorphisms of $\mathbb{P}_{k}^{2}$ of
the form 
\[
[x:y:z]\mapsto\left[x:y+s(\frac{x^{2}}{q_{0}})x:z-2ys(\frac{x^{2}}{q_{0}})-xs(\frac{x^{2}}{s_{0}})^{2}\right],\qquad s\in k[t]
\]
and of the elements of the subgroup $\mathrm{Aut}(\mathbb{P}^{2},Q_{0},p_{0})$
of $\mathrm{Aut}(\mathbb{P}_{k}^{2},Q_{0})$ consisting of automorphisms
fixing the point $p_{0}\in Q_{0}$.
\end{prop}
\begin{proof}
1) The assertion already appeared in \cite[\S 4.1.3]{DuLa} in the
case $k=\mathbb{C}$. It extends readily to the case of an arbitrary
base field $k$ thanks to the techniques developed in \cite{BD11},
so we just sketch the argument for the convenience of the reader.
Every automorphism $\varphi$ of $S_{0}$ uniquely extends to a birational
map $\overline{\varphi}:\mathbb{P}_{k}^{2}\dashrightarrow\mathbb{P}_{k}^{2}$.
If $\overline{\varphi}$ is biregular then it is an automorphism preserving
$S_{0}$, hence its complement $Q_{0}$, and so $\overline{\varphi}\in\mathrm{Aut}(\mathbb{P}_{k}^{2},Q_{0})$.
Otherwise if $\overline{\varphi}$ is strictly birational, it lifts
in a unique way to a strictly birational endormorphism of the $1$-standard
completion $(X_{0},B_{0})$ of $S_{0}$ by a chain $Q_{0}\vartriangleright C\vartriangleright E$
of type $[0,-1,-2,-2,-2]$ obtained by taking the minimal resolution
of $\overline{\rho}_{0}:\mathbb{P}_{k}^{2}\dashrightarrow\mathbb{P}_{k}^{1}$.
By virtue of Theorem 3.0.2 in \cite{BD11} this lift factors in a
unique way into a finite sequence of two particular types of birational
maps between $1$-standard completions $(X_{i},B_{i}=Q_{i}\vartriangleright C_{i}\vartriangleright E_{i})$
of $S_{0}$, called \emph{fibered modifications} and \emph{reversions}.
In our case, a fibered modification $(X_{i-1},B_{i-1})\dashrightarrow(X_{i},B_{i})$
descends through the contractions of the rational sub-chains $C_{i-1}\vartriangleright E_{i-1}$
and $C_{i}\vartriangleright E_{i}$ in $X_{i-1}$ and $X_{i}$ onto
$k$-rational points $x_{i-1}\in Q_{i-1}$ and $y_{i}\in Q_{i}$ to
a birational endomorphism $\overline{\varphi}_{i}:(\mathbb{P}_{k}^{2},Q_{i-1})\dashrightarrow(\mathbb{P}_{k}^{2},Q_{i})$
with the following properties:

a) $x_{i-1}$ and $y_{i}$ are the unique proper base points of $\overline{\varphi}_{i}$and
$\overline{\varphi}_{i}^{-1}$ respectively,

b) $\overline{\varphi}_{i}$ maps the pencil $\mathcal{P}_{x_{i-1}}\subset\left|\mathcal{O}_{\mathbb{P}_{k}^{2}}(2)\right|$
generated by the smooth $k$-rational conic $Q_{i-1}$ and $2T_{x_{i-1}}Q_{i-1}$
onto the pencil $\mathcal{P}_{y_{i}}\subset\left|\mathcal{O}_{\mathbb{P}_{k}^{2}}(2)\right|$
generated by the smooth $k$-rational conic $Q_{i}$ and $2T_{y_{i}}Q_{i}$,

c) $\overline{\varphi}_{i}$ restricts to an isomorphism between $\mathbb{P}_{k}^{2}\setminus Q_{i-1}$
and $\mathbb{P}_{k}^{2}\setminus Q_{i}$. 

\noindent On the other hand, the definition of a reversion $(X_{i-1},B_{i-1})\dashrightarrow(X_{i},B_{i})$
(see \cite[Definition 2.3.1]{BD11}) implies that such a map descends
via the same contractions to an isomorphism of pairs $(\mathbb{P}_{k}^{2},Q_{i-1})\stackrel{\sim}{\rightarrow}(\mathbb{P}_{k}^{2},Q_{i})$.
As a consequence, every strictly birational endomorphism $\overline{\varphi}:\mathbb{P}_{k}^{2}\dashrightarrow\mathbb{P}_{k}^{2}$
restricting to an automorphism of $S_{0}=\mathbb{P}_{k}^{2}\setminus Q_{0}$
admits a decomposition into a finite sequence of strictly birational
maps of pairs 
\[
\overline{\varphi}=\overline{\varphi}_{n}\circ\cdots\circ\overline{\varphi}_{2}\circ\overline{\varphi}_{1}:(\mathbb{P}_{k}^{2},Q_{0})\stackrel{\overline{\varphi}_{1}}{\dashrightarrow}(\mathbb{P}_{k}^{2},Q_{1})\stackrel{\overline{\varphi}_{2}}{\dashrightarrow}\cdots\stackrel{\overline{\varphi}_{n}}{\dashrightarrow}(\mathbb{P}_{k}^{2},Q_{n})=(\mathbb{P}_{k}^{2},Q_{0})
\]
satisfying properties a), b) and c) above. Now for every $i=1,\ldots,n$,
there exists an automorphism $\alpha_{i}:(\mathbb{P}_{k}^{2},Q_{0})\stackrel{\sim}{\rightarrow}(\mathbb{P}_{k}^{2},Q_{i-1})$
of $\mathbb{P}_{k}^{2}$ mapping $Q_{0}$ onto $Q_{i-1}$ and $p_{0}$
onto the proper base point $x_{i-1}$ of $\overline{\varphi}_{i}$
and an automorphism $\beta_{i}:(\mathbb{P}_{k}^{2},Q_{0})\stackrel{\sim}{\rightarrow}(\mathbb{P}_{k}^{2},Q_{i})$
mapping $Q_{0}$ onto $Q_{i}$ and $p_{0}$ onto the proper base point
$y_{i}$ of $\overline{\varphi}_{i}^{-1}$. \[\xymatrix{\cdots \ar@{-->}[r]^-{\overline{\varphi}_{i-1}} & (\mathbb{P}^2_k,Q_{i-1}) \ar@<-1ex>[d]_{\beta^{-1}_{i-1}} \ar@{-->}[r]^-{\overline{\varphi}_{i}} & (\mathbb{P}^2_k,Q_{i}) \ar@<-1ex>[d]_{\beta^{-1}_{i}} \ar@{-->}[r]^-{\overline{\varphi}_{i+1}}  & \cdots \\  \cdots \ar@{-->}[r] & (\mathbb{P}^2_k,Q_{0}) \ar@{-->}[r]^-{\overline{\psi}_{i}}  \ar@<-1ex>[u]_{\alpha_{i}} & (\mathbb{P}^2_k,Q_{0})\ar@<-1ex>[u]_{\alpha_{i+1}} \ar@{-->}[r] & \cdots }\]

The composition $\overline{\psi}_{i}=\beta_{i}^{-1}\circ\overline{\varphi}_{i}\circ\alpha_{i}$
is then a birational map of pairs $(\mathbb{P}_{k}^{2},Q_{0})\dashrightarrow(\mathbb{P}_{k}^{2},Q_{0})$
mapping the pencil $\mathcal{P}_{p_{0}}$ onto itself, and restricting
to an automorphism $\psi_{i}$ of $S_{0}=\mathbb{P}^{2}\setminus Q_{0}$
preserving the $\mathbb{A}^{1}$-fibration $\rho_{0}:S_{0}\rightarrow\mathbb{A}^{1}$
globally. Writing 
\begin{align*}
\overline{\varphi} & =\overline{\varphi}_{n}\circ\cdots\circ\overline{\varphi}_{2}\circ\overline{\varphi}_{1}=(\beta_{n}\circ\overline{\psi}_{n}\circ\alpha_{n}^{-1})\circ\cdots\circ(\beta_{2}\circ\overline{\psi}_{2}\circ\alpha_{2}^{-1})\circ(\beta_{1}\circ\overline{\psi}_{1}\circ\alpha_{1}^{-1})\\
 & =\beta_{n}\circ\overline{\psi}_{n}\circ(\alpha_{n}^{-1}\circ\beta_{n-1})\circ\cdots\circ(\alpha_{3}^{-1}\circ\beta_{2})\circ\overline{\psi}_{2}\circ(\alpha_{2}^{-1}\circ\beta_{1})\circ\overline{\psi}_{1}\circ\alpha_{1}^{-1}
\end{align*}
we obtain a decomposition of $\overline{\varphi}$ into an alternating
sequence of automorphisms $\beta_{n}$, $(\alpha_{i+1}^{-1}\circ\beta_{i})_{i=1,\ldots n-1}$,
$\alpha_{1}^{-1}$ of the pair $(\mathbb{P}_{k}^{2},Q_{0})$ and birational
endomorphisms $\overline{\psi}_{i}$ of $\mathbb{P}_{k}^{2}$ restricting
to elements of the group $\mathrm{Aut}(S_{0},\rho_{0})$. This shows
that $\mathrm{Aut}(S_{0})$ is generated by the subgroups $\mathrm{Aut}(\mathbb{P}_{k}^{2},Q_{0})|_{S_{0}}$
and $\mathrm{Aut}(S_{0},\rho_{0})$. The existence of an amalgamated
product structure follows from general properties of the above decompositions
into birational maps, see \cite[\S 3, Proposition 16]{DuLa} and \cite[Lemma 3.2.4]{BD11}. 

2) The description of the generators of $\mathrm{Aut}(\mathbb{P}_{k}^{2},Q_{0})$
follows from the classical faithful representation of $\mathrm{Aut}(Q_{0})=\mathrm{PGL}_{2}(k)$
as the special orthogonal group $\mathrm{SO_{3}(q_{0})}\subset\mathrm{GL}_{3}(k)$
of the quadratic form $q_{0}=xz+y^{2}$, defined by the action $\sigma:\mathrm{PGL}_{2}(k)\times T\rightarrow T$
of $\mathrm{PGL}_{2}(k)$ by conjugation on the space $T\simeq k^{3}$
of $2\times2$ matrices of trace zero. Explicitly, the representation
$\gamma:\mathrm{PGL}_{2}(k)\rightarrow SO_{3}(q_{0})$ is given by
\[
\mathrm{PGL}_{2}(k)\ni\begin{bmatrix}a & b\\
c & d
\end{bmatrix}\mapsto\frac{1}{ad-bc}\begin{bmatrix}a^{2} & -2ab & -b^{2}\\
-ac & ad+bc & bd\\
-c^{2} & 2cd & d^{2}
\end{bmatrix}\in\mathrm{SO}_{3}(q_{0}),
\]
and the listed generators of $\mathrm{Aut}(\mathbb{P}_{k}^{2},Q_{0})$
coincide with the respective images in $\mathrm{PGL}_{3}(k)$ of the
generators 
\[
\begin{bmatrix}1 & -b\\
0 & 1
\end{bmatrix},\,b\in k,\quad\begin{bmatrix}a & 0\\
0 & 1
\end{bmatrix},\,a\in k^{*}\quad\textrm{and }\begin{bmatrix}0 & 1\\
1 & 0
\end{bmatrix}
\]
of $\mathrm{PGL}_{2}(k)$ by $\gamma$. 

3) The generators of $\mathrm{Aut}(S_{0},\rho_{0})$ can be determined
as follows. The correspondence which maps every $\Psi\in\mathrm{Aut}(S_{0},\rho_{0})$
to the unique element $\psi_{\rho_{0}}\in\mathrm{Aut}(\mathbb{A}_{k}^{1})$
such that $\rho_{0}\circ\Psi=\psi_{\rho_{0}}\circ\rho_{0}$ defines
a group homomorphism $d:\mathrm{Aut}(S_{0},\rho_{0})\rightarrow\mathrm{Aut}(\mathbb{A}_{k}^{1})$.
Since $\rho_{0}^{-1}(0)$ is the unique degenerate fiber of $\rho_{0}$,
$\psi_{\rho_{0}}$ necessarily fixes the origin, hence belongs to
the sub-torus $\mathbb{G}_{m,k}\times\{0\}$ of $\mathrm{Aut}(\mathbb{A}_{k}^{1})=\mathbb{G}_{m,k}\ltimes\mathbb{G}_{a,k}$.
Conversely, the existence of the homomorphism $\mathbb{G}_{m,k}\rightarrow\mathrm{Aut}(S_{0},\rho_{0})\cap\mathrm{Aut}(\mathbb{P}_{k}^{2},Q_{0},p_{0})$,
$a\mapsto[ax:y:a^{-1}z]$ implies that we have a split exact sequence
\[
0\rightarrow\mathrm{Aut}_{0}(S_{0},\rho_{0})\rightarrow\mathrm{Aut}(S_{0},\rho_{0})\stackrel{d}{\rightarrow}\mathbb{G}_{m,k}\rightarrow0,
\]
and it remains to describe the elements of the group $\mathrm{Aut}_{0}(S_{0},\rho_{0})$
of automorphisms of $S_{0}$ preserving $\rho_{0}$ fiber wise. Every
such automorphism $\Psi$ restricts to an automorphism of the complement
of $\rho_{0}^{-1}(0)$ in $S_{0}$. Under the isomorphisms
\begin{align*}
S_{0}\setminus\rho_{0}^{-1}(0) & \simeq\mathbb{P}^{2}\setminus(Q_{0}\cup T_{p_{0}}Q_{0})\\
 & \simeq\mathrm{Spec}(k[Y,Z])\setminus\{Z+Y^{2}=0\}\\
 & \simeq\mathrm{Spec}(k[Y,t^{\pm1}]),
\end{align*}
where $Y=y/x$, $Z=z/x$ and $t=x^{2}q_{0}^{-1}=(Z+Y^{2})^{-1}$,
$\Psi|_{S_{0}\setminus\rho_{0}^{-1}(0)}$ coincides with a $\mathrm{Spec}(k[t^{\pm1}])$-automorphism
of $\mathrm{Spec}(k[Y,t^{\pm1}])$ which is thus of the form $(t,Y)\mapsto(t,\lambda t^{n}Y+s(t))$
for some $\lambda\in k^{*}$, $n\in\mathbb{Z}$ and $s(t)\in k[t^{\pm1}]$.
It follows that $\Psi$ is induced by the restriction of a birational
endomorphism $\overline{\Psi}$ of $\mathbb{P}_{k}^{2}$ of the form
\[
\left[x:y:z\right]\mapsto[x:\lambda(\frac{x^{2}}{q_{0}})^{n}y+s(\frac{x^{2}}{q_{0}})x:z+(1-\lambda^{2}(\frac{x^{2}}{q_{0}})^{2n})\frac{y^{2}}{x}-2\lambda(\frac{x^{2}}{q_{0}})^{n}s(\frac{x^{2}}{q_{0}})y-s(\frac{x^{2}}{q_{0}})^{2}x].
\]
 If $(1-\lambda^{2}(\frac{x^{2}}{q_{0}})^{2n})\neq0$ or $s(t)\in k[t^{\pm1}]\setminus k[t]$
then such a birational endomorphism $\overline{\Psi}$ contracts the
tangent line $T_{p_{0}}Q_{0}=\left\{ x=0\right\} $ to the point $[0:0:1]$,
hence is not the extension of any automorphism of $S_{0}$. So $n=0$,
$\lambda=\pm1$, $s(t)\in k[t]$ necessarily. Conversely, every $\overline{\Psi}$
of the form 
\[
[x:y:z]\mapsto\left[x:\lambda y+s(\frac{x^{2}}{q_{0}})x:z-2\lambda ys(\frac{x^{2}}{q_{0}})-xs(\frac{x^{2}}{s_{0}})^{2}\right]
\]
where $\lambda=\pm1$ and $s\in k[t]$ is the composition of an element
$[x:y:z]\mapsto[x:\pm y:z]$ of $\mathrm{Aut}(\mathbb{P}_{k}^{2},Q_{0},p_{0})$
and of a birational endomorphism of the desired type, which indeed
restricts to an automorphism of $S_{0}=\mathbb{P}^{2}\setminus Q_{0}$. 
\end{proof}

\subsection{Exotic affine lines in positive characteristic}

A well known consequence of the structure of $\mathrm{Aut}(\mathbb{A}_{k}^{2})$
is that if an embedded affine line $A\simeq\mathbb{A}_{k}^{1}$ in
$\mathbb{A}_{k}^{2}$ with parametrization $t\mapsto(x(t),y(t))$
belongs to the $\mathrm{Aut}(\mathbb{A}_{k}^{2})$-orbit of the coordinate
line $\left\{ x=0\right\} $, then either $\deg_{t}(x(t))$ divides
$\deg_{t}(y(t))$ or $\deg_{t}(y(t))$ divides $\deg_{t}(x(t))$.
Letting $L_{0}=\left\{ x=0\right\} $ and $L_{1}=\{x^{2}-q_{0}=0\}$
be the reduced fibers of the $\mathbb{A}^{1}$-fibration 
\[
\rho_{0}:S_{0}\rightarrow\mathbb{A}_{k}^{1}=\mathrm{Spec}(k[t]),\quad[x:y:z]\mapsto\frac{x^{2}}{q_{0}}
\]
over the closed points $0$ and $1$ of $\mathbb{A}_{k}^{1}$ respectively,
the description of $\mathrm{Aut}(S_{0})$ given in Proposition \ref{prop:Automorphisms}
leads to the following analogue for closed embeddings of $\mathbb{A}_{k}^{1}$
in $S_{0}$: 
\begin{lem}
\label{lem:degreduction} Let $j:\mathbb{A}_{k}^{1}\hookrightarrow S_{0}$,
$t\mapsto[x(t):y(t):z(t)]$ be a closed embedding with image $A$.
If $A$ belongs to the $\mathrm{Aut}(S_{0})$-orbit of $L_{0}$ or
$L_{1}$, then up to composition by the involution $[x:y:z]\mapsto[z:y:x]$,
the following hold:

a) $\deg_{t}(x(t))<\deg_{t}(y(t))<\deg_{t}(z(t)),$ 

b) If $\deg_{t}(x(t))\neq-\infty$, then it divides $\deg_{t}(y(t))$
and $\deg_{t}(z(t))$. 
\end{lem}
\begin{proof}
This holds for the parametrizations $t\mapsto[0:1:t]$ and $t\mapsto[1:t:1-t^{2}]$
of $L_{0}$ and $L_{1}$ respectively, and both properties are preserved
under the application of any of the generator of $\mathrm{Aut}(S_{0})$
listed in Proposition \ref{prop:Automorphisms}. 
\end{proof}
As a consequence of the above lemma, we obtain in every characteristic
$p\geq3$ the existence of closed embeddings $j:\mathbb{A}_{k}^{1}\hookrightarrow S_{0}$
whose image does not belong to the $\mathrm{Aut}(S_{0})$-orbit of
$L_{0}$ or $L_{1}$, a phenomenon similar to the failure of the Abhyankar-Moh
Theorem in positive characteristic. Namely, we have the following
family of examples:
\begin{prop}
\label{prop:poschar}Let $k$ be a field of characteristic $p\geq3$.
Then the morphism 
\[
j:\mathbb{A}_{k}^{1}\hookrightarrow S_{0},\,t\mapsto[t^{p^{2}}:t^{p^{2}}(t^{p^{2}+p}+t)+1:-t^{p^{2}}(t^{p^{2}+p}+t)^{2}-2(t^{p^{2}+p}+t)]
\]
is a closed embedding whose image does not belong to the $\mathrm{Aut}(S_{0})$-orbit
of $L_{0}$ or $L_{1}$. 
\end{prop}
\begin{proof}
Once we show that $j$ is indeed a closed embedding, the conclusion
follows immediately from Lemma \ref{lem:degreduction} above. Letting
$\tilde{S}_{0}\subset\mathbb{A}_{k}^{3}=\mathrm{Spec}(k[x,y,z])$
be the smooth affine surface with equation $xz+y^{2}-1=0$, $j$ is
the composition of the morphism 
\[
\tilde{j}:\mathbb{A}_{k}^{1}\rightarrow\tilde{S}_{0},\quad t\mapsto(t^{p^{2}},t^{p^{2}}(t^{p^{2}+p}+t)+1,-t^{p^{2}}(t^{p^{2}+p}+t)^{2}-2(t^{p^{2}+p}+t))
\]
with the \'etale Galois double cover $\pi:\tilde{S}_{0}\rightarrow S_{0}$,
$(x,y,z)\mapsto[x:y:z]$. Letting $\tilde{A}$ be the image of $\tilde{j}$,
$\pi^{-1}(\pi(\tilde{A}))$ is the disjoint union of $\tilde{A}$
with its image by the action $(x,y,z)\mapsto(-x,-y,-z)$ of the Galois
group of $\pi$. So $\pi$ induces an isomorphism between $\tilde{A}$
and the image of $j$. Since $\tilde{S}_{0}$ is affine, every embedding
of $\mathbb{A}_{k}^{1}$ into it is necessarily closed, and hence,
it now suffices to show that $\tilde{j}$ is an embedding. Noting
that $\tilde{A}$ is contained in the complement $V$ of the curve
with equation $\{x=y+1=0\}\subset\tilde{S}_{0}$ and that $V$ is
isomorphic to $\mathbb{A}_{k}^{2}$ via the restriction of the rational
map 
\[
\alpha:\tilde{S}_{0}\dashrightarrow\mathbb{A}_{k}^{2}=\mathrm{Spec}(k[x,v]),\,(x,y,z)\mapsto(x,x^{-1}(1-y)=(y+1)^{-1}z),
\]
we are reduced to check that $\alpha\circ\tilde{j}:\mathbb{A}_{k}^{1}\rightarrow\mathbb{A}_{k}^{2}$,
$t\mapsto(x(t),v(t))=(t^{p^{2}},t^{p^{2}+p}+t)$ is an embedding.
This follows from the identity $t^{p(p+1)}=(v(t)^{p}-x(t)^{p+1})^{p+1}$
which implies that the inclusion $k[x(t),v(t)]\subset k[t]$ is an
equality. 
\end{proof}
\begin{rem}
The morphism $\mathbb{A}_{k}^{1}\hookrightarrow\mathbb{A}_{k}^{2}$,
$t\mapsto(t^{p^{2}},t^{p^{2}+p}+t)$ used in the proof of the proposition
above is a typical example of closed embedding of the line in $\mathbb{A}_{k}^{2}$
whose image, as a consequence of the Jung and van der Kulk Theorem,
does not belong to the $\mathrm{Aut}(\mathbb{A}_{k}^{2})$-orbit of
the coordinate line $\{x=0\}$. 
\end{rem}
\bibliographystyle{amsplain} 

\providecommand{\bysame}{\leavevmode\hbox to3em{\hrulefill}\thinspace}
\providecommand{\MR}{\relax\ifhmode\unskip\space\fi MR }
\providecommand{\MRhref}[2]{%
  \href{http://www.ams.org/mathscinet-getitem?mr=#1}{#2}
}
\providecommand{\href}[2]{#2}
\begin{thebibliography}{}

\end{thebibliography}


\begin{thebibliography}{99}

\bibitem{AM75} S. Abhyankar and T. Moh, \emph{Embedding of the line in the plane},  J. Reine Angew.  Math.  276 (1975), 148-166.

\bibitem{BD11} J. Blanc and A. Dubouloz, \emph{Automorphisms of $\mathbb{A}^1$-fibered surfaces}, Trans. Amer. Math. Soc. 363 (2011), 5887-5924. 

\bibitem{Gi-Da1} V. I. Danilov and M. H. Gizatullin, \emph{Automorphisms of affine surfaces. I.},  Izv. Akad. Nauk SSSR Ser. Mat.  39  (1975), no. 3, 523-565.

\bibitem{Du05}  A. Dubouloz, \emph{Completions of normal affine surfaces with a trivial Makar-Limanov invariant}, Michigan Math. J.  52 (2004), no. 2, 289-308. 

\bibitem{Du06} A. Dubouloz, \emph{Embeddings of Danielewski surfaces in affine spaces}, Commentarii Mathematici Helvetici vol 81, no. 1, (2006), p. 49-73.

\bibitem{DuLa} A. Dubouloz and S. Lamy, \emph{Automorphisms of open surfaces with irreducible boundary}, Osaka J. Math. Volume 52, Number 3 (2015), 747-793.

\bibitem{Fu82} T. Fujita, \emph{On the topology of noncomplete algebraic surfaces}, J. Fac. Sci. Univ. Tokyo Sect. IA Math. 29 (1982), no. 3, 503-566.

\bibitem{Go69} J.E. Goodman,\emph{Affine open subset of algebraic varieties and ample divisors}, Ann. of Math., 89,  (1969), 160-183.

\bibitem{GM08} R.V. Gurjar, K. Masuda, M. Miyanishi and P. Russell, \emph{Affine Lines on Affine Surfaces and the Makar-Limanov Invariant},  Canad. J. Math. Vol. 60 (1), 2008 pp. 109-139.

\bibitem{Ii77} S. Iitaka,  \emph{On logarithmic Kodaira dimension of algebraic varieties}, Complex analysis and algebraic geometry, 175-189. Iwanami Shoten, Tokyo, 1977. 

\bibitem{Ju42} H.W.E. Jung, \emph{\"Uber ganze birationale Transformationen der Ebene}. J. {R}eine {A}ngew. Math.  184 (1942), 161-174.

\bibitem{La03} A. Langer, \emph{Logarithmic orbifold Euler numbers of surfaces with applications}. Proc. London Math. Soc. (3) 86 (2003), no. 2, 358-396.

\bibitem{MiyBook} M. Miyanishi, \emph{Open {A}lgebraic {S}urfaces}, CRM Monogr. Ser., 12, Amer. Math. Soc., Providence, RI, 2001.

\bibitem{Pal15} K. Palka, \emph{A new proof of the theorems of Lin-Zaidenberg and Abhyankar-Moh-Suzuki}, J. Algebra Appl. 14 (2015), No. 9, 1540012, 15 pp.

\bibitem{vdK53}  W. van der Kulk, \emph{On polynomial rings in two variables}, Nieuw Arch. Wisk., 1 (1953), 33-41.

\end{thebibliography}

\end{document}